\documentclass[11pt]{article} 
\usepackage{caption}
\usepackage[utf8]{inputenc} 
\usepackage{amsmath}
\usepackage{amssymb}
\usepackage{mathtools}%%%
\usepackage{leftindex}%%%%%
\usepackage{mhchem}%%%%
\usepackage{floatpag}
\usepackage{mathtools}

\usepackage[dvipsnames]{xcolor}
\usepackage[toc,page]{appendix}
\makeatletter
\newcommand*\bigcdot{\mathpalette\bigcdot@{.5}}
\newcommand*\bigcdot@[2]{\mathbin{\vcenter{\hbox{\scalebox{#2}{$\m@th#1\bullet$}}}}}
\makeatother

\newcommand{\gaussbin}[2]{\genfrac{[}{]}{0pt}{}{#1}{#2}_q}

\usepackage{ulem}

\usepackage{amsthm}
\usepackage{ mathrsfs }
\usepackage{placeins}
\usepackage{titlesec}
\titleformat{\paragraph}
{\normalsize\bfseries}{\theparagraph}{1em}{}
\titlespacing*{\paragraph}
{0pt}{3.25ex plus 1ex minus .2ex}{1.5ex plus .2ex}

\titleclass{\subsubsubsection}{straight}[\subsection]

\newcounter{subsubsubsection}[subsubsection]
\renewcommand\thesubsubsubsection{\thesubsubsection.\arabic{subsubsubsection}}
\renewcommand\theparagraph{\thesubsubsubsection.\arabic{paragraph}} 

\titleformat{\subsubsubsection}
  {\sffamily\small}{\thesubsubsubsection}{1em}{}
\titlespacing*{\subsubsubsection}
{0pt}{3.25ex plus 1ex minus .2ex}{1.5ex plus .2ex}

\newcommand{\RN}[1]{%
  \textup{\uppercase\expandafter{\romannumeral#1}}%
}
\usepackage{titlesec}
\titleformat{\section}{\normalfont\Large\bfseries}{\S\thesection}{1em}{}[]

\newtheorem{Theorem}{Theorem}[section]
\newtheorem{Corollary}[Theorem]{Corollary}
\newtheorem{Lemma}[Theorem]{Lemma}
\newtheorem{Definition}[Theorem]{Definition}
\newtheorem{Remark}[Theorem]{Remark}

\usepackage{float}
\restylefloat{table}
\usepackage{dblfloatfix}
\usepackage[hidelinks]{hyperref}
\usepackage{color}
\usepackage[table,xcdraw]{xcolor}

\usepackage{placeins}
  \usepackage{graphicx}
    \usepackage{caption}    
    \usepackage{array,booktabs} \usepackage{amsthm,amsmath,amsfonts,amssymb,latexsym,amscd,mathrsfs}
    \usepackage{textcomp,float}
    \usepackage[margin=25mm]{geometry}

    \usepackage{tikz}
    \usepackage{tikz-3dplot}
    \usetikzlibrary{shapes,calc,positioning}
\usepackage{afterpage}
\usepackage{placeins}
\usepackage{hyperref}
\def\bF{\mathbb{F}}
\def\Fq{\mathbb{F}_q}
\def\GL{\mathrm{GL}}
\def\PG{\mathrm{PG}}
\def\PGL{\mathrm{PGL}}
\usepackage{graphicx} 
\definecolor{amber(sae/ece)}{rgb}{1.0, 0.49, 0.0}
	\definecolor{darkcyan}{rgb}{0.0, 0.55, 0.55}
	\definecolor{darkseagreen}{rgb}{0.56, 0.74, 0.56}
	\definecolor{salmon}{rgb}{1.0, 0.55, 0.41}
\usepackage{booktabs} 
\usepackage{array} 
\usepackage{paralist} 
\usepackage{verbatim} 
\usepackage{subfig} 
\definecolor{ForestGreen}{RGB}{34,139,34}
\usepackage{fancyhdr} 
\definecolor{carminepink}{rgb}{0.92, 0.3, 0.26}
	\definecolor{electricyellow}{rgb}{1.0, 1.0, 0.0}
	\definecolor{chromeyellow}{rgb}{1.0, 0.65, 0.0}
	\definecolor{babyblue}{rgb}{0.54, 0.81, 0.94}

\usepackage{sectsty}
\allsectionsfont{\sffamily\mdseries\upshape} 

 \def\bF{\mathbb{F}}
\def\Fq{\mathbb{F}_q}
\def\GL{\mathrm{GL}}

\def\F2h{\mathbb{F}_{2^h}}

\def\PG{\mathrm{PG}}
\def\GL{\mathrm{GL}}

\def\PGL{\mathrm{PGL}}

\def\cP{\mathcal{P}}

\def\cV{\mathcal{V}}
\def\cZ{\mathcal{Z}}
\def\cC{\mathcal{C}}
\def\cH{\mathcal{H}}
\def\cN{\mathcal{N}}

\usepackage[nottoc,notlof,notlot]{tocbibind} 
\usepackage[titles,subfigure]{tocloft}

\title{Linear complete symmetric rank-distance codes}

\author{Nour Alnajjarine\footnote{University  of Rijeka (\texttt{nour.alnajjarine@math.uniri.hr})} \;  and \;
Michel Lavrauw\footnote{University of Primorska; Vrije Universiteit Brussel (\texttt{michel.lavrauw@famnit.upr.si})}}

\begin{document} 
\maketitle
\begin{abstract}
An $\Fq$-linear code of minimum distance $d$ is called complete if it is not contained in a larger $\Fq$-linear code of minimum distance $d$. In this paper, we classify $\Fq$-linear complete symmetric rank-distance (CSRD) codes in $M_{3\times 3}(\Fq)$ up to equivalence. This includes the classification of $\Fq$-linear maximum symmetric rank-distance (MSRD) codes in $M_{3\times 3}(\Fq)$. Our approach is mainly  geometric, and our results contribute towards the classification  of nets of conics in $\PG(2, q)$.

\end{abstract}
\textbf{Keywords:}
\hspace{0.2cm} 
Complete Codes,\hspace{0.2cm} Symmetric Rank-Distance Codes,\hspace{0.2cm} MRD Codes,\hspace{0.2cm} Nets of Conics

\section{Introduction}\label{SRDINTRO}

Let $M_{n\times m}(\Fq)$ be the set of $n\times m$ matrices defined over $\Fq$. A {\it rank-distance code} is a subset $C$ of $M_{n \times m}(\mathbb{F}_q)$ equipped with the rank-distance metric defined as 
$d(A,B)=rank(A-B)$ for $A, B \in C$. These codes were first introduced by Delsarte in \cite{Delsarte} and were studied by Gabidulin and Roth in \cite{Gabidulin,Roth}. They have been extensively studied due to their wide range of applications, including cryptography and network coding \cite{crypto,networkcoding}. \\
  
The {\it minimum distance} of $C$, denoted by $d(C)$ or simply $d$, is defined as  
\[
d(C) = \min \{d(X, Y) : X, Y \in C, X \neq Y\},
\]
where $|C|\geq 2$. For a rank-distance code $C$ with minimum distance $d$, the size of $C$ is bounded above by $ q^{n(m-d+1)}$ (the Singleton-like bound \cite{Delsarte}). A rank-distance code $C$ achieving this bound is called a {\it maximum rank-distance (MRD) code}. The existence of MRD codes for all {\it parameters} $n$, $m$, $q$ and $d$ was proved by Delsarte in \cite{Delsarte}. A new %construction 
family of MRD codes, alongside the classical {\it (generalised) Gabidulin codes}, can be found in \cite{sheekey2}. \\

Certain rank-distance codes, satisfying additional properties, have been studied in several papers. For example, in \cite{Schmidt, SymmMRD,Delsarte2,Schmidt1,Schmidt2}, rank-distance codes of symmetric, alternating, and Hermitian matrices were investigated. In this paper, we are interested in $\Fq$-linear {\it symmetric rank-distance (SRD) codes} defined as subspaces of the space of symmetric $n\times n$ matrices over $\Fq$. 
By \cite{Schmidt}, the dimension of an $\Fq$-linear SRD code $C$ is bounded above by 
\begin{eqnarray}\label{eqn:dim_bound}
 \dim(C)\leq  \begin{cases}
 \frac{n(n-d+2)}{2} &\text{if } n-d \text{ is even},\\
 \frac{(n+1)(n-d+1)}{2} &\text{if } n-d \text{ is odd}.\\
  \end{cases}
\end{eqnarray} 
An $\mathbb{F}_q$-linear SRD code that achieves this bound is called an {\it $\Fq$-linear maximum symmetric rank-distance (MSRD) code}. In \cite{SymmMRD}, a sphere-packing bound was established for SRD codes, and non-trivial  codes achieving this bound ({\it perfect codes}) were characterized as those having odd order $n$ and minimum distance 3.

\begin{Definition}
An $\Fq$-linear code of minimum distance $d$ is called {\rm complete} if it is not contained in a larger $\Fq$-linear code of minimum distance $d$ (in the same vector space).

\end{Definition}

Geometrically, an $\Fq$-linear rank-distance code in  $M_{n\times n}(\Fq)$ defines a subspace of the projective space $\PG(n^2- 1, q)$. The geometric setting for the study of MRD codes described below, was first studied in the context of semifields, see \cite{Lavrauw2008,Lavrauw2011}, which correspond to MRD codes with $d=n$. Note that the study of semifields predates the study of MRD codes, and much of the theory which was developed in the context of semifields can be carried over to the theory of MRD codes, see \cite{Lunardon2017,sheekey}.
The set of rank-one matrices in $M_{n\times n}(\Fq)$ coincides with the set of points of the {\it Segre variety } $\mathcal{S}_{n,n}(\Fq)$ in $\PG(n^2- 1, q)$. An $\Fq$-linear MRD code in $M_{n\times n}(\Fq)$ of minimum distance $d $ corresponds to an $(n(n-d+1)-1)$-dimensional subspace of $\PG(n^2 - 1, q)$ disjoint from the $(d-1)$-st secant variety of the Segre variety $\mathcal{S}^{(d-1)}_{n,n}(\Fq)$. For SRD codes, we can view an $\Fq$-linear MSRD code in $M_{n\times n}(\Fq)$ of minimum distance $d $ as a subspace of dimension determined by (\ref{eqn:dim_bound}), disjoint from the $(d-1)$-st secant variety of the {\it Veronese variety} $\cV_{n}(\Fq)$ in $\PG(N-1,q)$, where
$$N=\frac{n(n+1)}{2}.$$ We direct the reader to Section \ref{geometricInt} for further details on this geometric setting.\\

In this paper, we study $\Fq$-linear {\it complete symmetric rank-distance (CSRD)  codes} in $M_{3\times 3}(\Fq)$ of minimum distance $d$ as  subspaces of $\PG(5,q)$ which intersect the $(d-1)$-st secant variety of the Veronese variety $\cV_3^{(d-1)}(\Fq)$ trivially and which are maximal with respect to this property. Clearly, every MSRD code is complete, but we will show that the converse does not hold for $q$ even. Further, we classify $\Fq$-linear CSRD codes in $M_{3\times 3}(\Fq)$ with $d=2$ up to equivalence by employing geometric techniques based on the correspondence between these codes and linear systems of conics in $\PG(2, q)$.\\

 A {\it linear system of conics} in $\PG(2,q)$ is defined as a subspace of the projective geometry $\PG({\mathcal{F}}_2(2, q))$, where ${\mathcal{F}}_2(2, q)$ denotes the space of quadratic forms defined on the projective plane $\PG(2,q)$. Particularly, subspaces of $\PG({\mathcal{F}}_2(2,q))$ with dimensions 1, 2, 3, and 4 are referred to as pencils, nets, webs, and squabs of conics, respectively, see \cite{lines, solidsqeven, nets, planesqeven, webs}. \\

Herein, we classify $\Fq$-linear complete symmetric rank-distance (CSRD) codes in $M_{3\times 3}(\Fq)$, establishing the existence of $3$ (resp. $6$) $\Fq$-linear CSRD codes in $M_{3\times 3}(\Fq)$ with $d=2$ up to equivalence for $q$ odd (resp. even). 
\begin{Theorem}\label{mainodd}
     There are $3$ equivalence classes of $\Fq$-linear CSRD codes in $M_{3\times 3}(\Fq)$ of minimum distance $2$ for $q$ odd. Moreover, each $\Fq$-linear CSRD code in $M_{3\times 3}(\Fq)$, $q$ odd,  of minimum distance $2$ is an MSRD code.
 \end{Theorem}
\begin{Theorem}\label{maineven}
    There are $6$ equivalence classes of $\Fq$-linear CSRD codes in $M_{3\times 3}(\Fq)$ of minimum distance $2$ for $q\geq 4$ even, $3$ of which are  MSRD codes.
\end{Theorem}

This implies that, while every $\Fq$-linear MSRD code in $M_{3 \times 3}(\mathbb{F}_q)$ is complete, the converse does not necessarily hold when $q$ is even and $d=2$, and therefore each $\mathbb{F}_q$-linear SRD code in $M_{3 \times 3}(\mathbb{F}_q)$ can be extended to an MSRD code when $q$ is odd. When $q$ is even, there exist $3$-dimensional CSRD codes in $M_{3 \times 3}(\mathbb{F}_q)$ that cannot be extended to MSRD codes. Furthermore, in Corollary \ref{C1}, we show that the non-trivial perfect $\Fq$-linear CSRD codes in $M_{3 \times 3}(\mathbb{F}_q)$ correspond to nets of conics in $\PG(2,q)$  that exclude singular conics. \\

Our results also contribute to the classification of nets of conics in $\mathrm{PG}(2, q)$. In Theorem  \ref{[0,q+1,0,q^2]}, we show that nets  with $q+1$ double lines and $q^2$ non-singular conics form one  $\PGL(3,q)$-orbit. Similarly, for nets of conics with a unique double line and $q^2+q$ non-singular conics (see Theorem \ref{[0,1,0,q^2+q]}). In Corollary \ref{[0,0,0,q^2+q+1]}, we conclude that nets with no singular conics form a unique $\PGL(3,q)$-orbit for $q$ even and two $\PGL(3,q)$-orbits for $q$ odd. Our results also imply that a net of conics with an empty base and at least one pair of real or conjugate imaginary lines necessarily contains a pencil of conics with an empty base (see Corollaries \ref{netsb0odd} and \ref{netsb0even}).\\

The paper is structured as follows. In Section \ref{pre}, we present the definitions and theory needed  for our main results including Theorems \ref{mainodd} and \ref{maineven}. The proofs of these results are provided in Section \ref{mainsec}. Specifically, the results concerning $\Fq$-linear CSRD codes in $M_{3 \times 3}(\Fq)$ with $d = 2$ (resp. $d = 3$) are discussed in Sections \ref{Subd=2} (resp. \ref{Subd=3}). In Section \ref{linearsystems}, we give the applications of our results to the study of nets of conics in $\PG(2,q)$. In Section \ref{Comparison with known constructions of MSRD codes}, we compare our results with a known construction of MSRD codes.

\section{Preliminaries}\label{pre}

This section outlines the definitions and theory required for our proofs. Some of the results are well-known and can be found in \cite{galois geometry,hirsch,Havlicek2003}, while others are more recent results from \cite{webs,solidsqeven,planesqeven,roots,lines,nets,LaPoSh2021,sheekey}.\\

The finite field of order $q$  is denoted by $\bF_q$, where $q = p^h$ and $p$ is a prime number. The subset of squares in $\bF_q$ is represented by $\square_q$. The $n$-dimensional projective space over $\bF_q$ is denoted by $\PG(n, q)$. 
As usual, the number of $(k-1)$-dimensional subspaces of $\PG(n-1,q)$ is denoted by the Gaussian binomial coefficient $\gaussbin{n}{k}$.
A {\it form on $\PG(n, q)$} is a homogeneous polynomial in the ring $\bF_q[X_0, \ldots, X_n]$. The algebraic variety in $\PG(n, q)$ defined as the common zero set of the forms $g_1, \ldots, g_k$ is written as $\cZ(g_1, \ldots, g_k)$.

 \subsection{Linear systems of conics}

 A {\it conic} $\mathcal{C}$ in $\PG(2,q)$ is defined as the zero locus $\cZ(f)$ of a quadratic form $f$ on $\PG(2,q)$. There are four types of conics in $\PG(2,q)$ up to projective equivalence: double lines, pairs of distinct (real) lines, pairs of conjugate (imaginary) lines over $\bF_{q^2}$, and non-singular conics.\\ 

Let ${\mathcal{F}}_2(2,q)$ denote the 6-dimensional vector space of $2$-forms defined on $\PG(2, q)$. A {\it linear system of conics} in $\PG(2,q)$ is
a subspace of the projectivisation $\PG({\mathcal{F}}_2(2,q))\cong \PG(5,q)$ of ${\mathcal{F}}_2(2,q)$. In particular, $1$, $2$, $3$ and $4$-dimensional subspaces are respectively referred to as {\it pencils, nets,  webs} and {\it squabs of conics}. We denote by $\mathscr{P}=\langle \mathcal{C}_1, \mathcal{C}_2\rangle$ the pencil of conics determined by $\mathcal{C}_1\neq \lambda \mathcal{C}_2$; $\lambda \in \Fq$. A similar notion is adopted to represent the other linear systems. The {\it base} of a linear system of conics is defined as the intersection points of the conics that generate the system. A {\it rank one} linear system of conics is defined as a linear system with at least one double line.

\subsection{The Veronese surface $\cV(\Fq)$ in $\PG(5,q)$}\label{egSolid}

The {\it Veronese surface} $\cV(\bF_q)$ is a 2-dimensional algebraic variety in $\PG(5,q)$ defined as the image of the {\it Veronese map}
\[
\nu: \PG(2,q)\rightarrow \PG(5,q) \quad \text{given by} \quad
(u_0,u_1,u_2) \mapsto (u_0^2,u_0u_1,u_0u_2,u_1^2,u_1u_2, u_2^2).
\]
In this setting it is convenient to use the following notation, consistent with \cite{lines,nets,LaPoSh2021}. A point $P=(y_0,y_1,y_2,y_3,y_4,y_5)$ of $\PG(5,q)$ is represented by a symmetric $3 \times 3$ matrix
\[
M_P=\begin{bmatrix} y_0&y_1&y_2\\
y_1&y_3&y_4\\
y_2&y_4&y_5  \end{bmatrix}.
\]
This representation can be generalized to any subspace of $\PG(5,q)$. For example, the plane $\pi$ spanned by the first three points of the standard frame in $\PG(5,q)$ is represented by 
$$\pi=\begin{bmatrix} x&y&z\\
y&\cdot&\cdot\\
z&\cdot&\cdot  \end{bmatrix},$$ where a “$\cdot$" denotes a zero. 
We also refer to the tables in the Appendix for additional examples of representations of solids and lines in $\PG(5, q)$. 
Throughout the paper, the indeterminates of the coordinate rings of $\PG(2, q)$ and $\PG(5, q)$ are denoted by $(X_0, X_1, X_2)$ and $(Y_0, \ldots, Y_5)$, respectively.
\\

The {\it rank of a point} $P$ in $\PG(5,q)$ is the rank of its associated matrix $M_P$. 
The set of points of rank $1$ in $\PG(5,q)$ coincides with the set of points in  $\mathcal{V}(\Fq)$. The {\it secant variety of $\cV(\Fq)$}, denoted by $\cV^{(2)}(\Fq)$, consists of points of rank at most $2$ in $\PG(5,q)$. The Veronese surface $\cV(\Fq)$ contains $q^2 + q + 1$ {\it conics}, which are defined as the images of lines in $\PG(2, q)$ under the Veronese map $\nu$ (\cite[Remark 2.1]{webs}). For any two points $P, Q \in \cV(\mathbb{F}_q)$, there exists a unique conic passing through them, given by
\[
\mathcal{C}(P, Q) := \nu\left(\langle \nu^{-1}(P), \nu^{-1}(Q) \rangle\right),
\]
and any two such conics intersect in exactly one point. \\

Each conic of $\PG(2, q)$ corresponds to the hyperplane section of $\cV(\Fq)$ defined by:
\begin{eqnarray}\label{eqn:delta}
 \delta:  \cZ\Big(\sum_{0\leqslant i \leqslant j \leqslant 2}a_{ij}X_iX_j\Big)\mapsto\mathcal{Z}(a_{00}Y_0+a_{01}Y_1+a_{02}Y_2+a_{11}Y_3+a_{12}Y_4+a_{22}Y_5).
\end{eqnarray}

A {\it conic plane in $\PG(5, q)$} is defined as a plane intersecting  $\cV(\mathbb{F}_q)$ in a conic of $\cV(\mathbb{F}_q)$. 
 When $q$ is even, all tangent lines to a conic $\mathcal{C}$ in $\mathcal{V}(\mathbb{F}_q)$ are concurrent, meeting at the {\it nucleus} of $\mathcal{C}$. The collection of all nuclei of conics in $\mathcal{V}(\mathbb{F}_q)$ defines a plane in $\PG(5, q)$, known as the {\it nucleus plane} $\pi_{\mathcal{N}}$.   
In the matrix representation, points lying in $\pi_{\mathcal{N}}$ are represented by symmetric $3 \times 3$ matrices with zeros along the main diagonal, i.e., $\pi_{\mathcal{N}} = \mathcal{Z}(Y_0, Y_3, Y_5)$. We denote by $\Sigma_{\mathcal{N}}$ the $K$-orbit $\{\pi_{\mathcal{\mathcal{N}}}\}$.
Every point $R$ of rank two in $\PG(5, q)$ is contained in a unique conic plane $\langle \mathcal{C}(R)\rangle$ in $\cV(\mathbb{F}_q)$. If $R$ lies on the secant $\langle P, Q\rangle$ with $P, Q \in \cV(\Fq)$ then $\mathcal{C}(R)=\mathcal{C}(P,Q)$. If $q$ is even and $R$ is contained in the nucleus plane, then $R$ is the nucleus of $\mathcal{C}(R)$. The {\it tangent lines to $\cV(\mathbb{F}_q)$} are the lines which are contained in a conic plane $\pi$ of $\mathcal{V}(\mathbb{F}_q)$ and which are tangent to the conic $\pi \cap \mathcal{V}(\mathbb{F}_q)$.  
The tangent lines of $\cV(\mathbb{F}_q)$ at any point $P \in \cV(\mathbb{F}_q)$ lie within {\it the tangent plane of $\cV(\mathbb{F}_q)$ at $P$}. Note that, for $q$ odd, there exists a polarity of $\PG(5, q)$ that maps the conic planes of $\cV(\Fq)$ onto the tangent planes of $\cV(\Fq)$ (see e.g. \cite[Theorem 4.25]{galois geometry}).

\subsection{The $\PGL(3,q)$-orbits of linear systems of conics}\label{orbitsofsystems}

There is a one-to-one correspondence between the $\PGL(3,q)$-orbits of linear systems of conics in $\PG(2,q)$ and the $K$-orbits of subspaces in $\PG(5,q)$, where $K$ is a subgroup of $\PGL(6,q)$, isomorphic to $\PGL(3,q)$, described below. The $K$-orbits of points, lines, planes and solids in $\PG(5,q)$ correspond to the $\PGL(3,q)$-orbits of squabs, webs, nets and pencils of conics in $\PG(2,q)$.
The action of the group $K \leqslant \PGL(6, q)$ on the subspaces of $\PG(5, q)$ is defined as the lift of $\PGL(3, q)$ via the Veronese map $\nu$. 
Specifically, $K$ is defined as $K := \alpha(\PGL(3, q))$, where $\alpha~:~\PGL(3,q)\rightarrow \PGL(6,q)$ is the group monomorphism defined as follows.
For a projectivity $\phi_A \in \PGL(3, q)$ defined by the matrix $A \in \GL(3, q)$, the action of the associated projectivity $\alpha(\phi_A) \in \PGL(6, q)$ on points of $\PG(5, q)$ is defined as
\[
\alpha(\phi_A): P \mapsto Q \quad \text{where} \quad M_Q = A M_P A^T,
\]
where $M_P$ and $M_Q$ are the matrices corresponding to the points $P$ and $Q$.  
The group $K$ preserves $\cV(\mathbb{F}_q)$, and if $q > 2$, the group $K$ is the {\it full} setwise stabiliser of $\cV(\mathbb{F}_q)$ in $\PGL(6, q)$. However, for $q = 2$, the full setwise stabiliser of $\cV(\mathbb{F}_2)$ is isomorphic to the symmetric group $\operatorname{Sym}(7)$.\\

\begin{Definition}\label{def:minimum_rank}
\textnormal{
The {\it rank distribution} of a subspace $W \subseteq \PG(5, q)$ is given by the $3$-tuple $[r_1, r_2, r_3]$, where $r_i$ denotes the number of points in $W$ of rank $i$. The {\it minimum rank} of $W$ is defined as the smallest rank among all points $A \in W$. 
}
\end{Definition}
The rank distribution of a subspace is invariant under the action of $K$. The following definition introduces stronger $K$-invariants, initially defined  in \cite{LaPoSh2021}. These invariants are essential for distinguishing $K$-inequivalent subspaces in $\PG(5, q)$.
 \begin{Definition}\label{orbitdist}
\textnormal{
Let $[W_1,W_2,\dots,W_m]$ be the list of $K$-orbits of $r$-dimensional subspaces in $\PG(n,q)$. 
The \textit{$r$-space orbit distribution} of a subspace $W$ of $\PG(n,q)$ is defined as $OD_r(W)=[w_1,w_2,\ldots,w_m]$, where $w_i=w_i(W)$ counts the number of elements of the orbit $W_i$ incident with $W$.
}
\end{Definition}

There are four $\PGL(3,q)$-orbits of squabs of conics corresponding to the $K$-orbits of points in $\PG(5, q)$ defined as: the orbit $\mathcal{P}_1 = \cV(\mathbb{F}_q)$ of rank-$1$ points, consisting of $q^2 + q + 1$ points, the orbit $\mathcal{P}_3$ of rank-$3$ points, containing $q^5 - q^2$ points, and two orbits of rank-$2$ points, whose descriptions depend on whether $q$ is even or odd. For $q$ even, the rank-$2$ orbits are as follows:  $\mathcal{P}_{2,n}$, consisting of the $q^2 + q + 1$ points in the nucleus plane $\pi_{\mathcal{N}}$, and  
    $\mathcal{P}_{2,s}$, consisting of the $(q^2 - 1)(q^2 + q + 1)$ points contained in conic planes but not in $\pi_{\mathcal{N}} \cup \cV(\mathbb{F}_q)$.  For $q$ odd, the rank-$2$ orbits are:  
 $\mathcal{P}_{2,e}$, consisting of $(q^2 + q + 1)q(q+1)/2$ exterior points in conic planes, and  
 $\mathcal{P}_{2,i}$, consisting of $(q^2 + q + 1)q(q-1)/2$ interior points in conic planes. The point-orbit distribution of a subspace $W \subset \PG(5, q)$ is given by:  
\[
OD_0(W) = 
\begin{cases} 
    [r_1, r_{2,n}, r_{2,s}, r_3], & \text{if } q \text{ is even}, \\ 
    [r_1, r_{2,e}, r_{2,i}, r_3], & \text{if } q \text{ is odd},
\end{cases}
\]
where:  
$r_1 = r_1(W)$ and $r_3 = r_3(W)$ are the numbers of rank-$1$ and rank-$3$ points in $W$, respectively,  
     $r_{2,n} = r_{2,n}(W)$ is the number of rank-$2$ points in $W \cap \pi_{\mathcal{N}}$,  
 $r_{2,s} = r_{2,s}(W)$ is the number of rank-$2$ points in $W \setminus \pi_{\mathcal{N}}$,  
    $r_{2,e} = r_{2,e}(W)$ is the number of exterior rank-$2$ points in $W$, and  $r_{2,i} = r_{2,i}(W)$ is the number of interior rank-$2$ points in $W$. Note that the point-orbit distribution of a subspace in $\PG(5,2)$ is not an invariant under $\operatorname{Sym}_7$ (the full setwise stabiliser of $\cV(\mathbb{F}_2)$), since the nucleus plane is not preserved under the action of $\operatorname{Sym}(7)$. \\ 

In a similar manner, one can define the {\it line-}, {\it plane-}, {\it solid-}, and {\it hyperplane-orbit distributions} of $W$ for $r = 1, 2, 3, 4$, respectively. These distributions are particularly useful for determining the $K$-orbit of a subspace in $\PG(5, q)$.  The line orbits for all $q$ were determined in \cite{lines}, based on the classification of tensors in $\Fq^2\otimes \Fq^3\otimes \Fq^3$ obtained in \cite{canonical}. For $q$ odd, the $K$-orbits of solids correspond to the $K$-orbits of lines in $\PG(5, q)$. The $K$-orbits of solids for $q$ even were classified in \cite{solidsqeven}. Planes in $\PG(5,q)$ intersecting $\cV(\Fq)$ non-trivially were classified in \cite{nets} for $q$ odd and \cite{planesqeven} for $q$ even. As a consequence, these results led to the classifications of webs, pencils and non-empty base nets of conics in $\PG(2,q)$. Note that, rank-1 points of a subspace $W$ in $\PG(5,q)$ are the images under the Veronese map of the base points of its associated linear system \cite[Lemma 4.1]{planesqeven}. Therefore, planes in $\PG(5, q)$ disjoint from $\cV(\Fq )$ correspond to nets of
conics in $\PG(2, q)$ with empty bases. For $q$ odd, nets of
conics with non-empty bases correspond to nets of conics containing at least one double line (rank-one nets). \\

Hyperplanes of $\PG(5, q)$ are associated with conics in $\PG(2, q)$ via the map $\delta$ as detailed in (\ref{eqn:delta}). Specifically, there are four $K$-orbits of hyperplanes, defined as follows: $\cH_1$ corresponds to the $\PGL(3, q)$-orbit of double lines in $\PG(2, q)$,  
$\cH_{2,r}$ corresponds to the $\PGL(3, q)$-orbit of pairs of real lines, $\cH_{2,i}$ corresponds to the $\PGL(3, q)$-orbit of pairs of conjugate imaginary lines, and $\cH_3$ corresponds to the $\PGL(3, q)$-orbit of non-singular conics. 
Alternatively, these $K$-orbits of hyperplanes can be described based on their intersections with $\cV(\mathbb{F}_q)$:  $\cH_1$ consists of hyperplanes intersecting $\cV(\mathbb{F}_q)$ in a conic, $\cH_{2,r}$ consists of hyperplanes intersecting $\cV(\mathbb{F}_q)$ in two conics, $\cH_{2,i}$ consists of hyperplanes intersecting $\cV(\mathbb{F}_q)$ in a single point, and $\cH_3$ consists of hyperplanes intersecting $\cV(\mathbb{F}_q)$ in a normal rational curve of degree $4$. The shorthand notation $OD_4(W) = [h_1, h_{2,r}, h_{2,i}, h_3]$ is often used to represent the hyperplane-orbit distribution of a subspace $W$ in $\PG(5, q)$. Here, $h_1 = h_1(W)$ denotes the number of hyperplanes in $\cH_1$ that are incident with $W$, and similarly for the other components. The invariant $OD_4(W)$ is also referred to as the {\it conic distribution} of the linear system of conics associated with $W$.

\subsection{Geometric interpretation of linear  rank-distance codes}\label{geometricInt}

The set of rank-one matrices in $M_{n\times n}(\Fq)$ coincides with the set of points of the {\it Segre variety } $\mathcal{S}_{n,n}(\Fq)$ in $\PG(n^2- 1, q)$, which can be defined as the image of the {\it Segre embedding}:

$$
\sigma_{n,n}: \PG(n-1,q) \times \PG(n-1,q) \longrightarrow \PG(n^2-1,q) 
$$
\[
\quad ((x_0, \dots, x_{n-1}), (y_0, \dots, y_{n-1})) \longmapsto 
\left( x_0y_0, x_0y_1, \dots, x_iy_j, \dots, x_{n-1}y_{n-1} \right).
\]

 The {\it  linear equivalence} of the $\Fq$-linear rank-distance codes in $M_{n\times n}(\Fq)$, under the action of $\GL(n,q)\times \GL(n,q)$ on $M_{n\times n}(\Fq)$: $(A,B): X\mapsto AXB$, corresponds to the equivalence of the subspaces of $\PG(n^2- 1, q)$ under the action of the induced group $G$, as a subgroup of the projectivity group $\PGL(n^2,q)$, that stabilises the Segre variety $\mathcal{S}_{n,n}(\Fq)$ in $\PGL(n^2, q)$. 
 \begin{Remark}
   Note that there are two ways of slightly extending the above action of the group $G$ to more general notions of equivalence. One (called ``semi-linear equivalence") is obtained by extending this action of the group $G$ by automorphisms of $\Fq$, i.e. $X\mapsto AX^\tau B$, $\tau \in {\mathrm{Aut}}(\Fq)$. However, it is readily verified that this does not alter our results. The other extension is obtained by also considering the transpose operation, i.e. $X\mapsto X^T$, but since we are interested in $SRD$ codes, this action is not a faithful action, and can be discarded.  
 \end{Remark}

 For $1\leq i< n$, let $\mathcal{S}^{(i)}_{n,n}(\Fq)$ denote the $i$-th secant variety of the Segre variety consisting of points in $\PG(n^2 - 1, q)$ which correspond to matrices of rank at most $i$. Clearly, \begin{eqnarray*}
     \mathcal{S}^{(1)}_{n,n}(\Fq)=\mathcal{S}_{n,n}(\Fq),~ \mathcal{S}^{(i)}_{n,n}(\Fq)\subseteq \mathcal{S}^{(i+1)}_{n,n}(\Fq), \mbox{ and }~~\mathcal{S}^{(n)}_{n,n}(\Fq)=\PG(n^2-1,q).
\end{eqnarray*}
An $\Fq$-linear MRD code in $M_{n\times n}(\Fq)$ of minimum distance $d $ corresponds to an $(n(n-d+1)-1)$-dimensional subspace of $\PG(n^2 - 1, q)$ disjoint from the $(d-1)$-st secant variety of the Segre variety $\mathcal{S}^{(d-1)}_{n,n}(\Fq)$, i.e., {\it a subspace of minimum rank $d$}. For SRD codes, we can view an $\Fq$-linear MSRD code in $M_{n\times n}(\Fq)$ of minimum distance $d $ as a subspace of dimension determined by (\ref{eqn:dim_bound}), disjoint from the $(d-1)$-st secant variety of the {\it Veronese variety} $\cV_{n}(\Fq)$ in $\PG(N-1,q)$, which is defined as the intersection of the Segre variety with the subspace $\PG(N-1,q)$ of $\PG(n^2-1,q)$ corresponding to symmetric matrices. Note that for the $\mathbb{F}_q$-linear SRD codes in $M_{n \times n}(\mathbb{F}_q)$, we can define another type of equivalence, called \textit{symmetric equivalence}. Two codes are said to be {\it symmetrically equivalent} if their associated subspaces in $\mathrm{PG}( N - 1, q )$ are equivalent under the action of the subgroup of $G$ stabilizing the Veronese variety $\mathcal{V}_{n}(\mathbb{F}_q)$. The {\it rank distribution} of an $\Fq$-linear rank distance code $C$ in $M_{n\times n}(\Fq)$ is an $(n+1)$-tuple $(r_0,r_1,\cdots,r_n)$ where $r_i$ is the number of matrices in $C$ of rank $i$. Clearly, the rank distribution of an $\Fq$-linear SRD code in $M_{3\times 3}(\Fq)$ is completely determined by the rank distribution of its associated subspace in $\PG(5,q)$.

\section{$\Fq$-linear CSRD codes in $M_{3\times 3}(\Fq)$}\label{mainsec}

The main goal of this section is to classify $\Fq$-linear CSRD codes in $M_{3\times 3}(\Fq)$ up to equivalence. While every $\Fq$-linear MSRD code in $M_{3\times 3}(\Fq)$ is complete, we show that the converse does not hold when $d=2$ and $q$ is even. This leads to concluding that all $\Fq$-linear SRD codes in $M_{3\times 3}(\Fq)$ are extendable to MSRD codes when $q$ is odd. However, for $q$ even, the only $\Fq$-linear SRD codes in $M_{3\times 3}(\Fq)$ that cannot extend to an MSRD code are the $3$-dimensional CSRD codes with $d=2$. We recall that a $k$-dimensional SRD code in $M_{3\times 3}(\Fq)$ of minimum distance $d$ can be viewed as a $(k-1)$-dimensional subspace of $\PG(5,q)$ that is disjoint from $\cV^{(d-1)}(\Fq)$, i.e., of minimum rank $d$. Note that $\cV^{(1)}(\Fq)=\cV(\Fq)$ and  $\cV^{(0)}(\Fq)=\emptyset$. The cases $d=2$ and $d=3$ are addressed separately. The results in this section are initially proved   geometrically and then interpreted in their equivalent coding-theoretic form.

\subsection{Minimum distance $d=2$}\label{Subd=2}

Geometrically, an $\Fq$-linear CSRD code in $M_{3\times 3}(\Fq)$ of minimum distance 2 is a subspace of $\PG(5, q)$ of minimum rank $2$ that is maximal for this property. In particular, an $\Fq$-linear MSRD code in $M_{3\times 3}(\Fq)$ of minimum distance $2$ is a solid of $\PG(5,q)$ of minimum rank $2$. To study $\Fq$-linear CSRD codes in $M_{3\times 3}(\Fq)$ of minimum distance 2 that are not maximum, it suffices to examine lines and planes in $\PG(5,q)$ of minimum rank $2$ that are not contained in a larger subspace of minimum rank $2$. We treat these cases separately, leading to the classification of $\Fq$-linear CSRD codes in $M_{3\times 3}(\Fq)$ of minimum distance $2$.

\subsubsection{3-dimensional SRD codes}
We start by investigating the $3$-dimensional CSRD codes with $d=2$ in $M_{3 \times 3}(\Fq)$. Such codes correspond to planes of $\PG(5,q)$ of minimum rank $2$ which are not contained in a solid of minimum rank 2.

\begin{Remark}
The solids in $\PG(5, q)$ of minimum rank $2$ are the solids in $\Omega_{8,2} \cup \Omega_{14,2} \cup \Omega_{15,2}$ for $q$ odd, and in $\Omega_{7} \cup \Omega_{13} \cup \Omega_{14}$ for $q$ even. This follows from the point-orbit distributions of the $K$-orbits of solids in $\PG(5,q)$ as presented in Tables \ref{solidseven} and \ref{solidsodd}, based on \cite{lines,solidsqeven,webs}.
\end{Remark}

\begin{Theorem}\label{maximalplanesqodd}
    For $q$ odd, a plane in $\PG(5,q)$ of minimum rank $2$ is contained in a solid in $\Omega_{8,2}\cup\Omega_{14,2}\cup\Omega_{15,2}$. 

\end{Theorem}
\begin{proof}
Let $\pi$ be a plane in $\PG(5,q)$ of minimum rank $2$ and $P\in \pi\cap \cV^{(2)}(\Fq)$. Since the number of solids through $\pi$ in $\PG(5,q)$ is equal to the number of points on $\cV(\Fq)$, it follows that there must be at least one solid through $\pi$ which does not meet $\cV(\Fq)$, unless all solids of the form $\langle \pi,Q\rangle$, with $Q\in \cV(\Fq)$ are distinct. To see that this is not the case, consider the conic $\cC(P)$ and a secant line $\langle Q,Q'\rangle$ of $\cC(P)$ through $P$. Then $\langle \pi,Q\rangle=\langle \pi,Q'\rangle$.
\end{proof}

\begin{Theorem}\label{nonmaximalplaneseven}
    For $q\geq 4$ even, a plane in $\PG(5,q)$ of minimum rank $2$ with at least one rank-$2$ point outside the nucleus plane $\pi_{\mathcal{N}}$ is contained in a solid in $\Omega_{7}\cup\Omega_{13}\cup\Omega_{14}$.
\end{Theorem}
\begin{proof}
   A similar proof to that of Theorem \ref{maximalplanesqodd} applies here. Note that the existence of a secant line through $ P $ is guaranteed, as  $ \pi $ contains at least one rank-2 point outside $ \pi_{\mathcal{N}} $. 
    \end{proof}

Next, we turn our attention to planes of minimum rank two in $\PG(5,q)$, $q$ even, which have no points of rank two outside of the nuclues plane.
We start by classifying planes in $\PG(5,q)$, $q$ even, which intersect the nucleus plane in a line, and which have no other points of rank $\leq 2$. The following theorem establishes the existence of a unique $K$-orbit of such planes, and determines its hyperplane-orbit distribution. 

\begin{Remark}\label{naming}
The orbits in Theorems \ref{[0,q+1,0,q^2]} and \ref{[0,1,0,q^2+q]} are named to align with the naming convention in \cite{planesqeven, netswith//qeven}, where $\Sigma_1,\dots,\Sigma_{15}$ represent the $K$-orbits of planes in $\PG(5,q)$ intersecting $\cV(\Fq)$ non-trivially for $q$ even. The orbits $\Sigma_{16}$ and $\Sigma_{17}$ in \cite{netswith//qeven} denote the $K$-orbits of planes in $\PG(5,q)$ that intersect $\pi_{\mathcal{N}}$ in a line, while the orbits $\Sigma_{18}, \dots, \Sigma_{23}$ denote the $K$-orbits of planes in $\PG(5,q)$ that intersect $\pi_{\mathcal{N}}$ in a point.

\end{Remark}

\begin{Theorem}\label{[0,q+1,0,q^2]}
    For $q\geq 4$ even, planes in $\PG(5,q)$ with point-orbit distribution $[0,q+1,0,q^2]$ form a unique $K$-orbit, namely $\Sigma_{16}$, with hyperplane-orbit distribution $[q+1,0,0,q^2]$.
    \end{Theorem}
    
    \begin{proof}

 A plane of $\PG(5,q)$  intersecting $ \pi_{\mathcal{N}}$ in a line has the form 
 $$\pi_{a,b,c}=\langle \ell_{12,1},P(0,0,a,b,0,c)\rangle=\begin{bmatrix}\cdot&x&az\\x&bz&y\\ az&y&cz\end{bmatrix},$$ 
 where $(a,b,c)\in \Fq^3\setminus \{(0,0,0)\}$, and
 $\ell_{12,1}$ is the representative of the $K$-orbit $o_{12,1}$ from Table \ref{tableoflineseven}.  For $c \neq 0$, we have $\pi_{0,b,c} \cap \cV(\Fq) \neq \emptyset$ 
 (particularly $\pi_{0,b,c}\in \Sigma_8$ from \cite[Table 1]{planesqeven}) and $\pi_{1,b,c}$ intersects $\cV^{(2)}(\Fq)$ in $\ell_{12,1}$ and $q$ rank-$2$ points outside $\pi_{\mathcal{N}}$. For $c = 0$, $\pi_{1,0,0} = \pi_{\mathcal{N}}$, and $\pi_{0,1,0} \cap \cV(\Fq) \neq \emptyset$  (particularly $\pi_{0,1,0}\in \Sigma_7$ from \cite[Table 1]{planesqeven}). The plane $\pi_{1,b,0}$ intersects $\cV^{(2)}(\Fq)$ in $\ell_{12,1}$ for all $b \neq 0$, and it is $K$-equivalent to $\pi_{1,1,0}$ for all $b \neq 0$. This equivalence can be seen by defining $$X=\begin{bmatrix}b^{-1}&0&0\\0&b^{-1}&0\\ 0&0&1\end{bmatrix}$$
 and observing that $$X\pi_{1,b,0}X^T=\begin{bmatrix}\cdot&b^{-2}x&b^{-1}z\\b^{-2}x&b^{-1}z&b^{-1}y\\ b^{-1}z&b^{-1}y&\cdot\end{bmatrix}.$$
 This shows that planes of $\PG(5,q)$ intersecting $\cV^{(2)}(\Fq)$ in a line in $\pi_{\mathcal{N}}$ define a unique $K$-orbit with point-orbit distribution $[0, q+1, 0, q^2]$. The associated net of conics with $\pi_{1,1,0}$ is defined by $\mathcal{N}(\pi_{1,1,0})=(X_0^2,X_0X_2+X_1^2,X_2^2)$. The singular conics of $\mathcal{N}(\pi_{1,1,0})$ are the zero locus of $$aX_0^2+b(X_0X_2+X_1^2)+cX_2^2$$ with $b=0$. Thus, $\mathcal{N}(\pi_{1,1,0})$ comprises $q+1$ double lines and $q^2$ non-singular conics.  \end{proof}

Our next objective is to classify planes in $\PG(5,q)$, $q$ even, which intersect the nucleus plane in a point, and which have no other points of rank $\leq 2$. The classification will be obtained in Theorem \ref{[0,1,0,q^2+q]}. The proof is more complicated than in the previous case, and we first need to prove some auxiliary results.

\begin{Lemma}\label{rank2pointino161}
The subgroup in $K$ stabilizing a rank-$2$ point in $\pi_{\mathcal{N}}$ is $E_q^2: \mathrm{GL}(2,q)$, where $E_q$ is the elementary abelian group of order $q$.
\end{Lemma}
\begin{proof}
Without loss of generality, let $P=(0,0,0,0,1,0)\in \pi_{\mathcal{N}}$. An element of $K$ with underlying matrix $g=(g_{ij})\in \mathrm{GL}(3,q)$ stabilises $P$ if and only if  $g_{12}=g_{13}=0$ and $g_{23}g_{32}+g_{22}g_{33}\neq 0$. Therefore, the subgroup $K_P$ stabilizing $P$ in $K$ is $E_q^2: \mathrm{GL}(2,q)$, as there are no restrictions on $g_{21}$ and $g_{31}$.
\end{proof}

The following theorem is needed for the proof of the main result of this section (Theorem \ref{[0,1,0,q^2+q]}), but it is also interesting in its own right.

\begin{Theorem}\label{OnePointInN}
For every plane $\pi$ in $\PG(5,q)$, $q\geq 4$ even, we have $r_{2,n}(\pi)=h_1(\pi)$, i.e., the number of hyperplanes in $\mathcal{H}_1$ containing $\pi$ is equal to $\lvert \pi \cap \pi_{\mathcal{N}} \rvert$. 
\end{Theorem}
\begin{proof}

First, note that $\pi_{\mathcal{N}}$ is contained in each hyperplane of type $ \cH_1$ (see \cite[Theorem 3.3]{webs}, in particular the last line of part (i) of its proof). Thus, $h_1(\pi_{\mathcal{N}})=q^2+q+1$. Next, if $r_{2,n}(\pi)=0$, then $h_1(\pi)=0$, since otherwise there would be a hyperplane $H \in \cH_1$ containing both $\pi$ and $\pi_{\mathcal{N}}$, implying $|\pi\cap \pi_{\mathcal{N}}|\geq 1$. 

We have two remaining possibilities to consider: either $\pi$ intersects $\pi_{\mathcal{N}}$ in a line or in a point. 

In the former case, we have $r_{2,n}(\pi)=q+1$, and $\pi=\langle \ell, P\rangle$, where $\ell\subset \pi_{\mathcal{N}}$ and $P$ is a rank-$1$ point, a rank-$2$ point outside $\pi_{\mathcal{N}}$, or a rank-$3$ point. By \cite[Table 2]{webs}, $P$ lies on exactly $q+1$ hyperplanes in $\cH_1$, and since $\ell \subset H$ for all $H \in \cH_1$, it follows that $h_1(\pi)=q+1$. 

To prove that $h_1(\pi)=1$ for planes meeting $\pi_{\mathcal{N}}$ in a point $P$, first observe that $h_1(\pi)\leq 1$, since if $\pi$ were contained in two distinct hyperplanes $H_1, H_2 \in \cH_1$, then $\pi$ and $\pi_{\mathcal{N}}$ are both contained in the solid $H_1\cap H_2$, contradicting the fact that $\pi$ meets $\pi_{\mathcal{N}}$ in exactly a point.

Next, we count flags $(\pi, H) $ with $ \pi_{\mathcal{N}}\cap \pi=\{P\}$ and $H \in \cH_1$. If $\alpha$ denotes the number of planes in a hyperplane $H\in \cH_1$ which meet $\pi_{\mathcal{N}}$ in exactly one point (observe that $\alpha$ is independent of the choice of $H\in \cH_1$), and $\beta$ denotes the number of planes in $\PG(5,q)$ which meet $\pi_{\mathcal{N}}$ in exactly one point, then
$$\alpha=\gaussbin{4}{2}
-(q+1)\left ( \gaussbin{3}{1}-1 \right )-1,$$
and
$$\beta=\gaussbin{5}{2}-(q+1)\left(\gaussbin{4}{1}-1\right)-1.$$
If $\pi_1,\ldots,\pi_\beta$ denote the planes in $\PG(5,q)$ which meet $\pi_{\mathcal{N}}$ in a point, then this gives
$$
\sum_{i=1}^\beta h_1(\pi_i)=(q^2+q+1)\alpha.
$$
Since $h_1(\pi_i)\leq 1$, for each $i\in \{1,\ldots,\beta\}$, this implies that $h_1(\pi)=1$, for each $i\in \{1,\ldots,\beta\}$.
\end{proof}

The next lemma gives a geometric property of the planes of $\PG(5,q)$, $q>2$ even, belonging to the $K$-orbit $\Sigma_{11}$, as described in \cite{planesqeven}. Such a plane meets both the Veronese surface and the nucleus plane in exactly one point. We will need this lemma in the proof of Theorem \ref{[0,1,0,q^2+q]}.

\begin{Lemma}\label{Sigma_11}
    If $\pi$ is a plane of $\PG(5,q)$ belonging to the $K$-orbit $\Sigma_{11}$, $q>2$ even, $Q_2=\pi\cap\pi_\cN$, and $H(Q_2)$ is the unique hyperplane of $\cH_1$ containing the conic $\cC(Q_2)$, then $\pi$ is not contained in $H(Q_2)$.
\end{Lemma}
\begin{proof}
    Let $\pi$ be a plane of $\PG(5,q)$ belonging to the $K$-orbit $\Sigma_{11}$, $q>2$ even. Recall that $OD_0(\pi)=[1,1,q-1,q^2]$. If for $\pi$ we take the representative corresponding to the matrix
    $$M_\pi=\begin{bmatrix}x&y&0\\y&z&z\\ 0&z&x+z\end{bmatrix},$$
then $Q_2$ is parametrized by $(x,y,z)=(0,1,0)$ (see \cite{planesqeven}), and the hyperplane $H(Q_2)\in \cH_1$ containing the conic $\cC(Q_2)$ is $H(Q_2)=\cZ(Y_5)$. Clearly $\pi$ is not contained in $H(Q_2)$.
\end{proof}

The following theorem classifies planes in $\PG(5,q)$, $q>2$ even, which intersect the nucleus plane in a point, and which have no other points of rank $\leq 2$, establishing the existence of a unique $K$-orbit of such planes. Moreover, it is shown that this $K$-orbit has hyperplane-orbit distribution $[1,0,0,q^2+q]$. For reasons explained in Remark \ref{naming} this orbit is denoted by $\Sigma_{18}$.

\begin{Theorem}\label{[0,1,0,q^2+q]}
    For $q\geq 4$ even, planes in $\PG(5,q)$ with point-orbit distribution $[0,1,0,q^2+q]$ form a unique $K$-orbit, namely $\Sigma_{18},$ with hyperplane-orbit distribution $[1,0,0,q^2+q]$.
    \end{Theorem}
       
   \begin{proof}
     
     Suppose $\pi$ is such a plane, and  $P=\pi\cap \pi_{\mathcal{N}}$. Then, by \cite{lines}, all lines in $\pi$ through $P$ belong to $o_{16,1}$, and the remaining $q^2$ lines belong to $o_{17}$. Also, if $K_{\pi}$ and  $K_P$ denote the stabilisers in $K$ of $\pi$ and $P$, then, $K_{\pi}\subset K_P=E_q^2: \mathrm{GL}(2,q)$ (see Lemma \ref{rank2pointino161}).\\
     
   (i) First we prove that $OD_4(\pi)=[1,0,0,q^2+q]$.

   By Theorem \ref{OnePointInN}, $\pi$ is contained in a unique hyperplane in $\cH_1$. If $\pi \subset H$ for some $H \in \cH_{2,r}$ where  $H\cap \cV(\Fq)=\mathcal{C}_1\cup \mathcal{C}_2$, $\pi_1=\langle \mathcal{C}_1\rangle $ and $\pi_2=\langle \mathcal{C}_2\rangle$, then $\pi \cap \pi_1 \cap \pi_2=\{P\}$. Thus,  $\pi_1=\pi_2$, a contradiction. This shows that $h_{2,r}(\pi)=0$.
     
     Next, suppose that $\pi \subset H$ for some $H \in \cH_{2,i}$ with $H\cap \cV(\Fq)=\{R\}$. Then $H(\bF_{q^2})$ meets $\cV(\mathbb{F}_{q^2})$ in two conjugate conics $\cC$ and $\cC^q$ over the quadratic extension $\bF_{q^2}$ of $\Fq$.
    Let $\rho=\langle \cC \rangle$. Then $\rho\cap \rho^q=\{R\}$, and the only points of rank two in $H(\bF_{q^2})$ are contained in $\rho\cup \rho^q$. Hence $P\in \rho \cup \rho^q$. Since $P^q=P$, this implies that $P=\rho \cap \rho^q$, contradicting $\rho\cap \rho^q=\{R\}$.
     
     Therefore, $\pi$ is not contained in a hyperplane in $\mathcal{H}_{2,r}$, and  the hyperplane-orbit distribution of $\pi$ is $[1,0,0,q^2+q]$, as required.\\
     
     (ii) Next we show that there is at most one $K$-orbit consisting of planes with point-orbit distribution $[0,1,0,q^2+q]$. 
     
     Again, suppose $\pi$ is such a plane and  $P=\pi\cap \pi_{\mathcal{N}}$. By Theorem \ref{OnePointInN}, we know that $\pi$ is contained in a unique hyperplane $H_1 \in \cH_1$. Let $\mathcal{C} = H_1 \cap \cV(\Fq)$. Then $H_1$ contains both $\pi$ and the conic plane $\langle \mathcal{C} \rangle$, which must intersect in at least one point. Since $\pi \cap \cV^{(2)}(\Fq) = \{P\}$, it follows that $\mathcal{C} = \mathcal{C}(P)$. Thus, $H_1 = H(P)$, where $H(P) \cap \cV(\Fq) = \mathcal{C}(P)$.
Consider the stabiliser $K_P$ of $P$ in $K$. Note that $K_P=K_{H(P)}$. Since a line $\ell_{17}$ of type $o_{17}$ is contained in a unique hyperplane in $\cH_1$ (by \cite[Lemma 4.26]{webs}), it follows that $K_{\ell_{17}}\subset K_P$ for each line $\ell_{17}$ in $o_{17}$ contained in $H(P)$. By \cite{lines}, $|K_{\ell_{17}}|=3$. Therefore the orbit of $\ell_{17}$ under $K_P$ has size
$$
\frac{|K_P|}{|K_{\ell_{17}}|}=\frac{q^2|\GL(2,q)|}{3}=\frac{1}{3}q^3(q-1)(q^2-1).
$$
By \cite[Theorem 4.29]{webs} this is equal to the total number of lines in $o_{17}$ contained in $H(P)$.
Hence $K_P$ acts transitively on the lines in $H(P)$ which belong to $o_{17}$.

     Next we claim that each plane $\langle P,\ell_{17}\rangle$, where $\ell_{17}$ is a line of type $o_{17}$ contained in $H(P)$, is disjoint from $\cV(\Fq)$. For if not, then $\langle P,\ell_{17}\rangle$ has point-orbit distribution $[a_1\geq 1,a_{2,n}\geq 1,a_{2,s},a_3\geq q+1]$. By \cite{planesqeven}, this implies that
$$
\langle P,\ell_{17}\rangle \in \Sigma_{3} \cup\Sigma_{4} \cup\Sigma_{8} \cup\Sigma_{9} \cup\Sigma_{10} \cup\Sigma_{11} \cup \Sigma_{15}.
$$     
Let $\eta$ denote a plane belonging to one of these orbits $\Sigma_i$, $i\in \{3,4,8,9,10,11,15\}$, and consider the cubic curve $\cC_\eta$ obtained by intersecting $\eta$ with the secant variety $\cV^{(2)}(\Fq)$. A direct computation shows in each case the cubic $\cC_\eta$ contains a line as a component, except for $\eta\in \Sigma_{11}$. Clearly, if $\cC_\eta$ contains a line then $\eta$ cannot contain a line $\ell_{17}$ of type $o_{17}$ since all points on $\ell_{17}$ have rank three. This implies that $\langle P,\ell_{17}\rangle \in \Sigma_{11}$. By Lemma \ref{Sigma_11} the plane $\langle P,\ell_{17}\rangle$ does not belong to $H(P)$. This proves the claim.\\

(iii) To finish the proof, we still need to show that such planes exist.
Note that it suffices to show that a plane $\pi=\langle P,\ell_{17}\rangle$, where $P\in \pi_{\mathcal{N}}$, and $\ell_{17}$ is a line of type $o_{17}$ in $H(P)$, with $H(P)$ the unique hyperplane in $\cH_1$ through $P$, has exactly one point of rank 2, since we already showed in part (ii) above that such a plane $\pi$ is disjoint from $\cV(\Fq)$. By way of contradiction suppose $\pi$ has at least two points of rank two. Since $\pi$ contains a constant rank three line, by the above, it follows that $\pi$ must have point-orbit distribution 
$[0,1,a_{2,s}\geq 1,a_3\geq q+1]$. Let $Q$ be a point of rank two in $\pi$ which does not lie in the nucleus plane. Then the line $\langle P,Q\rangle$ must have point-orbit distribution $[0,1,l_{2,s}\geq 1,l_3\geq 1]$. By \cite{lines} this implies, that the line $\langle P,Q\rangle$ belongs to $o_{13,1}$ and therefore belongs to the $K$-orbit of the line $\ell$ represented by
$$\begin{bmatrix}\cdot&x&\cdot\\x&y&\cdot\\ \cdot&\cdot &y\end{bmatrix}.$$
If $g\in K$ maps  $\langle P,Q\rangle$ to $\ell$ then $P^g$ is parameterised by $(x,y)=(1,0)$ and $Q^g$ is parameterised by $(x,y)=(0,1)$. But the point $Q^g$ does not belong to the hyperplane $H(P^g)$. This is a contradiction since it implies that $Q$ does not belong to $H(P)$. Therefore $\pi$ has a unique point of rank two.           
    \end{proof}

The following theorem establishes the existence of maximal planes of minimum rank two in $\PG(5,q)$, $q$ even. This illustrates a fundamental difference between $q$ even and $q$ odd. The proof relies on Theorem \ref{[0,q+1,0,q^2]} and Theorem \ref{[0,1,0,q^2+q]}, and the classification of solids from \cite{solidsqeven}.

\begin{Theorem}\label{CSRDplanesqeven}
        For $q\geq 4$ even, a plane in $\PG(5,q)$ disjoint from $\cV(\Fq)$ and intersecting  $\cV^{(2)}(\Fq)$ in a nonempty subspace of $\pi_{\mathcal{N}}$ is maximal with respect to that property, i.e., it is not contained in any solid of minimum rank 2.
\end{Theorem}

\begin{proof}

Recall that there are 3 $K$-orbits of solids of minimum rank two in $\PG(5,q)$, $q$ even: $\Omega_{7}$, $\Omega_{13}$, and $\Omega_{14}$. First observe that the nucleus plane $\pi_{\mathcal{N}}$ is maximal as a subspace of minimum rank $2$, since 
solids in $\Omega_{7}\cup\Omega_{13}\cup\Omega_{14}$ intersect $\pi_{\mathcal{N}}$ in at most $q+1$ points (see Table \ref{solidseven}). 
   Let $\pi\neq \pi_{\mathcal{N}}$ be a plane in 
$\PG(5,q)$ of minimum rank $2$ with no  rank-$2$ points outside $\pi_{\mathcal{N}}$.

If $\pi$ intersects $\pi_{\mathcal{N}}$ in a line, then $\pi \in \Sigma_{16}$ and $OD_0(\pi)=[0,q+1,0,q^2]$ (see Theorem \ref{[0,q+1,0,q^2]}). If $\pi$ intersects $\pi_{\mathcal{N}}$ in a point, then $\pi \in \Sigma_{18}$ and $OD_0(\pi)=[0,1,0,q^2+q]$ (see Theorem \ref{[0,1,0,q^2+q]}).
Moreover, it follows from the hyperplane-orbit distributions (determined in Theorem \ref{[0,q+1,0,q^2]} and Theorem \ref{[0,1,0,q^2+q]}) that in each of these two cases, the net of conics $\cN(\pi)$ corresponding to $\pi$ contains no line pairs (real or imaginary). Since each of the pencils of conics corresponding to the solids in $\Omega_{7}\cup\Omega_{13}\cup\Omega_{14}$ have at least one pair of conjugate imaginary lines, it follows that the net $\cN(\pi)$ contains no such pencil.
Therefore, planes in $\Sigma_{16}\cup \Sigma_{18}$ are maximal subspaces of minimum rank $2$. 
\end{proof}

\begin{Corollary}
Planes in $\PG(5,q)$ of minimum rank $2$  extend to solids in $\PG(5,q)$ of minimum rank $2$, except for planes intersecting $\cV^{(2)}(\Fq)$ in a non-empty subspace of $\pi_{\mathcal{N}}$ when $q$ is even, i.e., planes in $\Sigma_{\cN}\cup\Sigma_{16}\cup\Sigma_{18}$.
\end{Corollary}
\begin{proof}
    This is a direct consequence of Theorems \ref{maximalplanesqodd},  \ref{nonmaximalplaneseven} and \ref{CSRDplanesqeven}
\end{proof}
\begin{Corollary}
   Every $3$-dimensional $\Fq$-linear SRD code in $M_{3\times 3}(\Fq)$ of minimum distance $2$ extends to an MSRD code, except for those with rank distributions $(1,0,q^3-1,0)$, $(1,0,q^2-1,q^3-q^2)$ and $(1,0,q-1,q^3-q)$ for $q>2$ even.
\end{Corollary}
\begin{proof}
    A $3$-dimensional $\Fq$-linear SRD code $C$ in $M_{3\times 3}(\Fq)$ of minimum distance $2$ corresponds to a plane $\pi$ in $\PG(5,q)$ of minimum rank 2. By Theorem \ref{maximalplanesqodd} such a code $C$ is contained in an MSRD code in $M_{3\times 3}(\Fq)$, if $q$ is odd. If $q>2$ is even, then the same holds true, unless $\pi$ intersects $\cV^{(2)}(\Fq)$ in a non-empty subspace of $\pi_{\mathcal{N}}$, by the above Corollary. The rank distribution of the code $C$ follows from the point-orbit distributions of the three $K$-orbits of planes, corresponding to the three different cases in the proof of Theorem \ref{CSRDplanesqeven}. The nucleus plane has point-orbit distribution $[0,q^2+q+1,0,0]$, and the other two point-orbit distributions were determined in Theorem \ref{[0,q+1,0,q^2]} and Theorem \ref{[0,1,0,q^2+q]}.
\end{proof}

\subsubsection{2-dimensional SRD codes}
In this section, we demonstrate the non-existence of $2$-dimensional CSRD codes in $M_{3\times 3}(\Fq)$ with minimum distance $2$ by examining lines in $\PG(5,q)$ of minimum rank $2$. In fact, we will show that each line  of minimum rank 2 is contained in a solid of minimum rank two. As a corollary, it follows that each $2$-dimensional $\Fq$-linear SRD code in $M_{3\times 3}(\Fq)$ of minimum distance $2$ extends to an MSRD code in $M_{3\times 3}(\Fq)$.

\begin{Remark}\label{nomaxlinesandhyperplanes}
 A line $\ell$ in $ \PG(5,q)$ of minimum rank $2$ belongs to $o_{10}\cup o_{12,1}\cup o_{13,1}\cup o_{13,2}\cup o_{14,1}\cup o_{14,2} \cup o_{15,1}\cup o_{15,2}\cup o_{16,1}, $ for $q$ odd, and to $o_{10}\cup o_{12,1}\cup o_{12,3}\cup o_{13,1}\cup o_{13,3} \cup o_{14,1}\cup o_{15,1}\cup o_{16,1}\cup o_{16,13}, $ for $q$ even (see Tables \ref{tableoflineseven} and \ref{solidsodd}).

 \begin{Lemma}\label{NoCSRDLines}
    A line $\ell$ in $ \PG(5,q)$ of minimum rank $2$ is contained in a plane of minimum rank 2.
 \end{Lemma}
 \begin{proof}
     This follows by observing that a line $\ell$ of minimum rank $2$ is contained in $q^3 + q^2 + q + 1$ planes in $\PG(5, q)$, at most $q^2 + q + 1$ of which intersect $\cV(\Fq)$ non-trivially. Hence, $\ell$ is contained in at least $q^3$ planes in $\PG(5, q)$ of minimum rank $2$. 
 \end{proof}
 
\begin{Corollary}
 There are no $2$-dimensional CSRD codes in $M_{3\times 3}(\Fq)$ with $d=2$.
\end{Corollary}

\end{Remark}

\begin{Lemma}\label{maxlineseven}
A line in $\PG(5,q)$ of minimum rank $2$ extends to a solid in $\Omega_{8,2}\cup\Omega_{14,2}\cup\Omega_{15,2}$ for $q$ odd and $\Omega_{7}\cup\Omega_{13}\cup\Omega_{14}$ for $q$ even.
\end{Lemma}
\begin{proof}

Let $\ell$ be a line in $\PG(5,q)$ of minimum rank $2$. By Lemma \ref{NoCSRDLines}, $\ell$ is contained in at least $q^3$ planes of minimum rank $2$. For $q$ odd, it follows by Lemma \ref{maximalplanesqodd} that $\ell \in \Omega_{8,2}\cup\Omega_{14,2}\cup\Omega_{15,2}$. If $q$ is even and $\ell \cap \cV^2(\Fq)\not\subset \pi_{\mathcal{N}}$, then $\ell \in \Omega_{7}\cup\Omega_{13}\cup\Omega_{14}$ by Lemma \ref{nonmaximalplaneseven}. If $q$ is even and $\ell \cap \cV^2(\Fq)$ is a non-empty subspace of  $\pi_{\mathcal{N}}$, then $\ell \in o_{12,1}\cup o_{16,1}$ by \cite{lines}. Next, we claim that a solid in $\Omega_7$ contains at least one line in  $o_{12,1}$ and at least one line in $ o_{16,1}$. Let $S_7$ be the representative of $\Omega_7$ as described in Table \ref{solidseven},
$$
S_7:\begin{bmatrix} x&y&z\\ y&x+\gamma y&t\\z&t&y \end{bmatrix}.
$$
The point-orbit distribution of $S_7$ is $[0,q+1,q^2+q,q^3-q]$, and therefore $S_7$ intersects $\pi_{\mathcal{N}}$ in a line belonging to $o_{12,1}$. Consider the line $L$ in $S_7$ parameterised by $(x,y,z,t)=(0,y,0,t)$:
$$
L:\begin{bmatrix} 0&y&0\\ y&\gamma y&t\\0&t&y \end{bmatrix}.
$$
A direct computation shows that $L$ has a point-orbit distribution $OD_0(L)=[0,1,0,q]$. By \cite{lines} (see Table \ref{tableoflineseven}), this implies that $L$ belongs to the $K$-orbit $o_{16,1}$.
\end{proof}

\begin{Corollary}
   Every $2$-dimensional $\Fq$-linear SRD code in $M_{3\times 3}(\Fq)$ of minimum distance $2$ extends to an MSRD code.
\end{Corollary}

\subsubsection{The number of equivalence classes of CSRD codes}

In the following, we determine the number of equivalence classes of $\Fq$-linear CSRD codes in $M_{3\times 3}(\Fq)$ with $d=2$. Recall the definition of the group $G$ defined at the start of Section \ref{geometricInt}.
Also recall the geometric setup where the Veronese variety is embedded in the Segre variety, ${\mathcal{V}}(\Fq)\subset {\mathcal{S}}_{3,3}(\Fq)$, $\PG(5,q)=\langle {\mathcal{V}}(\Fq)\rangle$, and $\PG(8,q)=\langle {\mathcal{S}}_{3,3}(\Fq)\rangle$.
Following \cite{lines}, we say that a $G$-orbit $W^G$ of a subspace $W$ of $\PG(8,q)$ is said to {\it admit a symmetric representation} if there exists a subspace $U$ of $\PG(5,q)$ such that $U^G=W^G$. If $U$ and $W$ are subspaces of $\PG(5,q)$, then the $K$-orbits $U^K$ and $W^K$ are said to {\it merge} as $G$-orbits in $\PG(8,q)$ if $U^G=W^G$.

 \begin{Theorem}\label{equivqoddgeom}
  For $q$ odd, there are three $G$-orbits of solids in $\PG(8, q)$ that admit a symmetric representation with minimum rank $2$.
 \end{Theorem} 
 \begin{proof}
     It follows from Remark \ref{nomaxlinesandhyperplanes},  Theorem \ref{maximalplanesqodd} and the point-orbit distributions of solids in $\PG(5,q)$ (summarized in Table \ref{solidsodd}) that solids in $\Omega_{8,2}\cup\Omega_{14,2}\cup\Omega_{15,2}$ are the only solids of minimum rank $2$. These $3$ $K$-orbits of solids in $\PG(5,q)$ do not merge as $G$-obits in the embedded space $\PG(8,q)$, since they are completely identified by their rank distributions. 
 \end{proof}

 \begin{Theorem}
     There are $3$ equivalence classes of $\Fq$-linear CSRD codes in $M_{3\times 3}(\Fq)$ of minimum distance $2$ for $q$ odd. 
 \end{Theorem}
 \begin{proof}
    This follows immediately from Theorem \ref{equivqoddgeom}.
 \end{proof}

 \begin{Remark}
     For $q$ odd, each $\Fq$-linear CSRD code of minimum distance $2$ in $M_{3\times 3}(\Fq)$ is an MSRD code.
 \end{Remark}

\begin{Theorem}\label{equivqevengeom}
For $q\geq 4$ even, there are $six$ $G$-orbits of subspaces in $\PG(8, q)$ that admit a symmetric representation which has minimum rank $2$, and is maximal with this property.
\end{Theorem}
\begin{proof}
If follows from Lemma \ref{maxlineseven}, Theorems  \ref{nonmaximalplaneseven} and \ref{CSRDplanesqeven} that solids in $\Omega_7\cup\Omega_{13}\cup \Omega_{14}$ and planes in $\Sigma_{\mathcal{N}}\cup \Sigma_{16}\cup \Sigma_{18}$ are the only maximal subspaces in $\PG(5,q)$ of minimum rank $2$. 
 These subspaces define six distinct $G$-orbits in $\PG(8,q)$, since they are completely identified by their rank distributions. 
\end{proof}

\begin{Theorem}
    There are $6$ equivalence classes of $\Fq$-linear CSRD codes in $M_{3\times 3}(\Fq)$ of minimum distance $2$ for $q\geq 4$ even, $3$ of which are  MSRD codes.
\end{Theorem}
\begin{proof}
    Follows by Theorem \ref{equivqevengeom}. Note that, solids in $\Omega_7\cup\Omega_{13}\cup \Omega_{14}$ correspond to the $\Fq$-linear MSRD codes of minimum distance $2$ in $M_{3\times 3}(\Fq)$.
\end{proof}

We end this section by summarizing the explicit correspondence between $\Fq$-linear CSRD codes in $M_{3\times 3}(\Fq)$ with $d=2$  and their equivalent types of pencils, nets, and webs of conics in $\PG(2,q)$. Remark that, Corollary \ref{C3.18} is a consequence of the correspondence between lines in $o_{j,2}$ and solids in $\Omega_{j,2}$ for $j\in \{8,14,15\}$ and $q$ odd.

\begin{Corollary}\label{pencilintrepqodd}
 For $q$ odd, $\Fq$-linear CSRD codes in $M_{3\times 3}(\Fq)$ of minimum distance $2$ correspond to pencils of conics in $\PG(2,q)$ with conic distributions $[1,0,1,q-1]$, $[0,1,2,q-2]$ and $[0,0,1,q]$. \end{Corollary}

\begin{Corollary}\label{pencilintrepqeven}
 For $q\geq 4$ even, $\Fq$-linear CSRD codes in $M_{3\times 3}(\Fq)$ of minimum distance $2$ correspond to pencils of conics in $\PG(2,q)$ with conic distributions $[1,0,1,q-1]$, $[0,1,2,q-2]$ and $[0,0,1,q]$ and nets of conics in $\PG(2,q)$ with conic distributions $[q+1,0,0,q^2]$, $[1,0,0,q^2+q]$ and $[q^2+q+1,0,0,0]$.
  
\end{Corollary}

 \begin{Corollary}\label{C3.18}

  For $q$ odd, $\Fq$-linear CSRD codes in $M_{3\times 3}(\Fq)$ of minimum distance $2$ correspond to webs of conics in $\PG(2,q)$ with conic distributions $[0,q^2+\frac{3q+1}{2},\frac{q+1}{2},q^3-q]$, $[0,\frac{q^2+1}{2}+2q,\frac{q^2+1}{2}+q,q^3-2q]$ and $[0,\frac{q^2+1}{2}+q,\frac{q^2+1}{2},q^3]$. Moreover, these are the only webs of conics in $\PG(2,q)$, $q$ odd,  that are not of rank one.
\end{Corollary}

\subsection{Minimum distance $d=3$}\label{Subd=3}

In this section we classify $\Fq$-linear CSRD codes in $M_{3\times 3}(\Fq)$ of minimum distance $3$. We will see that each such CSRD code is in fact an MRD code. This means that this case is closely related to Menichetti's proof of a conjecture by Kaplanski concerning finite semifields, as explained in the following remark. 

\begin{Remark}\label{rem:Menichetti}
Note that MSRD codes of minimum distance $3$ in $M_{3\times 3}(\Fq)$ have dimension $3$, and are therefore also MRD codes. The study of the geometry related to MRD codes,  used in this paper, predates the study of MRD codes, and goes back to papers \cite{Menichetti1977,Menichetti1996} by Menichetti from the 1970's and possibly earlier, which were set in the context of the study of finite semifields (non-associative algebras). MRD codes of minimum distance $n$ in $M_{n\times n}(\Fq)$ correspond to $n$-dimensional semifields, and the equivalence classes of MRD codes correspond to the isotopism classes of semifields, see e.g. \cite{Lavrauw2008,Lunardon2017}. In 1977 Menichetti proved that a 3-dimensional semifield is either a field or isotopic to a twisted field. This essentially classifies MRD codes in $M_{3\times 3}(\Fq)$. It remains to determine which of the $G$-orbits of planes of minimum rank 3 (and therefore constant rank 3) in $\PG(8,q)$ have a symmetric representative. 
Somewhat similar to Menichetti's approach are the two papers on symplectic semifield spreads of $\PG(5, q^2)$, \cite{CPeven} for $q $ even and  \cite{GPodd} for $q$ odd, which also contain the classification of $K$-orbits of planes of minimum rank 3.
\end{Remark}

 In what follows, we give a classification of CSRD codes of minimum distance $3$ in $M_{3\times 3}(\Fq)$, without relying on the results mentioned in Remark \ref{rem:Menichetti}.
By \cite{lines}, the lines in $\PG(5, q)$ of constant rank-$3$ points form a unique $K$-orbit, denoted by $o_{17}$. Consequently, each line in $o_{17}$ lies within a plane $\pi$ of minimum rank 3. Therefore, there are no $2$-dimensional $\Fq$-linear CSRD codes with minimum distance $3$ in $M_{3\times 3}(\Fq)$, and hence each $\Fq$-linear CSRD code in $M_{3 \times 3}(\Fq)$ with minimum distance $3$ is an MSRD code.  
Note that by \cite{SymmMRD}, these codes are perfect, establishing a correspondence between nets of conics in $\PG(2,q)$ with no singular conics and non-trivial perfect $\Fq$-linear MSRD codes in $M_{3\times 3}(\Fq)$.\\

\begin{Theorem}\label{thm:cst3}
If $q> 2$ is even then there is unique $K$-orbit of planes of minimum rank 3 in $\PG(5,q)$, while if $q$ is odd there are two such $K$-orbits.
Each such $K$-orbit has hyperplane-orbit distribution $[0,0,0,q+1]$.
\end{Theorem}
\begin{proof}
Consider a fixed cubic extension $\bF_{q^3}$ of $\Fq$. If $P$ is a point of $\PG(2,q^3)\setminus \PG(2,q)$ which is not contained in a secant line to $\PG(2,q)$, then the points $P$, $P^q$, and $P^{q^2}$ form a triangle in $\PG(2,q^3)$, and their images under the Veronese map generate a plane $\pi(q^3)$ of type $\Sigma_2$ in $\PG(5,q^3)$ (cf \cite{nets}). Since $\pi(q^3)$ is fixed by the projectivity $\sigma$ of $\PG(5,q^3)$ induced by the Frobenius automorphism $a\mapsto a^q$ of $\bF_{q^3}$, it follows that $\pi(q^3)$ intersects the subgeometry $\PG(5,q)$ in a plane $\pi$. The plane $\pi$ has minimum rank 3, since the points of rank $\leq 2$ in the plane $\pi(q^3)$ lie on the triangle with vertices $\nu_3(P)$, $\nu_3(P^q)$, and $\nu_3(P^{q^2})$, and none of these points is fixed by $\sigma$.
Moreover, since the group $\PGL(3,q)$ acts transitively on the set of points of $\PG(2,q^3)\setminus \PG(2,q)$ which are not contained in a secant line of $\PG(2,q)$, it follows that the group $K$ acts transitively on these planes. Let $\Sigma_{{\mathrm{GF}}(q^3)}$  denote the $K$-orbit in $\PG(5,q)$ of such planes.
Observe that, since a plane $\pi\in\Sigma_{{\mathrm{GF}}(q^3)}$ contains no points of rank $\leq 2$, a hyperplane of $\PG(5,q)$ containing $\pi$ cannot contain a conic plane, a tangent plane, or the nucleus plane (if $q$ is even). Therefore, each hyperplane containing $\pi$ belongs to $\cH_3$.

Suppose $q$ is odd, and let $\rho$ denote the polarity of $\PG(5, q)$ that maps the conic planes of $\cV(\Fq)$ onto the tangent planes of $\cV(\Fq)$. Since $OD_4(\Sigma_{{\mathrm{GF}}(q^3)})=[0,0,0,q^2+q+1]$, it follows that the image of $\pi\in\Sigma_{{\mathrm{GF}}(q^3)}$ under $\rho$ is a plane $\pi^\rho$ of minimal rank 3 in $\PG(5,q)$. Let $\Sigma_{{\mathrm{TF}}(q^3)}$ denote the $K$-orbit containing $\pi^\rho$ in $\PG(5,q)$.
Observe that, if $\pi(q^3)$ is the plane in $\PG(5,q^3)$ as described above, then $\rho$ maps the three points $\nu_3(P)$, $\nu_3(P^q)$, and $\nu_3(P^{q^2})$ to three hyperplane $H$, $H^q$, $H^{q^2}$ $\in \cH_1(q^3)$. These three hyperplanes are the images under $\delta$ (defined in (\ref{eqn:delta})) of the three sides $\ell$, $\ell^q$, $\ell^{q^2}$ of a triangle in $\PG(2,q^3)$ (considered as double line conics).

Thus far, we have constructed two $K$-orbits of planes $\Sigma_{{\mathrm{GF}}(q^3)}$  and $\Sigma_{{\mathrm{TF}}(q^3)}$ in $\PG(5,q)$. A plane in $\Sigma_{{\mathrm{GF}}(q^3)}$  is the span of three conjugate points $\nu_3(P)$, $\nu_3(P^q)$, and $\nu_3(P^{q^2})$ in $\cV_3(\bF_{q^3})\setminus \cV_3(\bF_{q})$, and a plane in $\Sigma_{{\mathrm{TF}}(q^3)}$ is the intersection of three conjugate hyperplanes $H$, $H^q$, $H^{q^2}$ $\in \cH_1(q^3)$. Note that the $K$-orbit $\Sigma_{{\mathrm{TF}}(q^3)}$ consists of planes of minimum rank 3 only when $q$ is odd, since if $q$ is even, each of the hyperplanes in $\cH_1(q^3)$ contains the nucleus plane $\pi_N$.

In the remainder of the proof we show that, besides $\Sigma_{{\mathrm{GF}}(q^3)}$  and $\Sigma_{{\mathrm{TF}}(q^3)}$, there are no other $K$-orbits of planes of minimum rank 3 in $\PG(5,q)$.

Consider any plane $\pi$ in $\PG(5,q)$ and suppose $\pi$ has minimum rank 3. Let $f\in \bF_q[X,Y,Z]$ be the cubic form defined by the discriminant of the net of conics corresponding to $\pi$ (cf. \cite{nets}). Since $\pi$ has minimum rank 3, it follows that $\cZ(f)$ has no $\bF_q$-rational points, and therefore $\cZ(f)$ is not absolutely irreducible (see e.g. \cite{hirsch}). Hence $\cZ(f)$ is the union of three conjugate lines $L$, $L^q$, $L^{q^2}$ forming a triangle in the plane $\pi(q^3)$ in $\PG(5,q^3)$. Since points of rank one are necessarily singular points of $\cZ(f)$ it follows that $\pi(q^3)$ has either 3 points of rank one (the vertices $Q$, $Q^q$, $Q^{q^2}$ of the triangle $L$, $L^q$, $L^{q^2}$) or no points of rank one. In the first case the plane $\pi$ belongs to $\Sigma_{{\mathrm{GF}}(q^3)}$ . So, suppose $\pi(q^3)$ has no points of rank one.  Let $H$ be the tangent hyperplane to $\cV^{(2)}(\bF_{q^3})$ at $Q$. Then $H$ contains the lines contained in $\cV^{(2)}(\bF_{q^3})$ through $Q$, and therefore $H$ contains $\pi(q^3)$. Hence in this case $\pi(q^3)$ is the intersection of three conjugate hyperplanes $H$, $H^q$, $H^{q^2}$ $\in \cH_1(q^3)$ in $\PG(5,q^3)$, i.e. $\pi \in \Sigma_{{\mathrm{TF}}(q^3)}$.
\end{proof}

\begin{Theorem}\label{Maind=3}
For $q>2$ even, there is a unique  equivalence class  of non-trivial perfect $\Fq$-linear MSRD codes in $M_{3\times 3}(\Fq)$, while for $q$ odd, there are two such equivalence classes.
\end{Theorem}
\begin{proof}
 This is immediate for $q>2$ even, since there is only one $K$-orbit of planes of minimum rank 3. For $q$ odd, it follows from the proof of Theorem \ref{thm:cst3} that the $K$-orbits $\Sigma_{{\mathrm{GF}}(q^3)}$  and $\Sigma_{{\mathrm{TF}}(q^3)}$ have different rank distributions in $\PG(5,q^3)$. It follows that their embeddings in $\PG(8,q^3)$  define two distinct orbits under the group $G(q^3)\leq \PGL(9,q^3)$ stabilising the Segre variety $\mathcal{S}_{3,3}(\bF_{q^3})$ in $\PG(8,q^3)$. By considering the embedding of $\PG(8,q)$ as a subgeometry of $\PG(8,q^3)$, the group $G$ becomes a subgroup of the group $G(q^3)$. It follows that the $K$-orbits $\Sigma_{{\mathrm{GF}}(q^3)}$  and $\Sigma_{{\mathrm{TF}}(q^3)}$ do not merge as $G$-orbits in $\PG(8,q)$.
\end{proof}

\begin{Corollary}\label{C1}
There is a one-to-one correspondence between nets of conics in $\PG(2,q)$ with no singular conics and non-trivial perfect $\Fq$-linear MSRD codes in $M_{3\times 3}(\Fq)$. 
\end{Corollary}

\begin{Remark}
  Using the FinInG package \cite{fining} in GAP \cite{GAP}, one can verify the existence of three $\PGL(3,2)$-orbits of planes and three $\PGL(3,2)$-orbits of solids in $\PG(5,2)$ that have minimum rank 2 and are maximal with respect to this property. Additionally, there is a unique $\PGL(3,2)$-orbit of planes with minimum rank 3 that is also maximal in this regard. 
  Whereas, the group preserving $\cV(\mathbb{F}_2)$ is \(\operatorname{Sym}_7\), which does not preserve the rank distribution as it merges the two $K$-orbits $\cP_{2,n}$ and $\cP_3$. The underlying reason is that a conic plane ($K$-orbit $\Sigma_1$ in \cite{planesqeven}, defined as the image of a line under the Veronese map) meets $\cV(\mathbb{F}_2)$ in just 3 points, and so it has the same number of rank one points as a plane in the $K$-orbit $\Sigma_2$ as defined in \cite{planesqeven}. The bigger group \(\operatorname{Sym}_7\) does not preserve conic planes, and merges the two orbits $\Sigma_1$ and $\Sigma_2$. For the same reason, the group $\operatorname{Sym}_7$ does not preserve the set of planes of constant rank 3 points. From a coding-theoretic perspective, both actions  
  lack significant relevance.\\

\end{Remark}

\section{Impact on classifying nets of conics in $\PG(2,q)$}\label{linearsystems}

In this section, we use Theorems \ref{maximalplanesqodd},  \ref{nonmaximalplaneseven}, \ref{OnePointInN} and the correspondence between the set of nets of conics in $\PG(2,q)$ with empty bases and the set of planes in $\PG(5,q)$ disjoint from $\cV(\Fq)$ (see \eqref{orbitsofsystems}) to conclude that a net of conics with an empty base and at least one pair of lines (real or imaginary) contains a pencil of conics with an empty base. This result provides a better  understanding of the configurations defining these nets, and thereby contributing to the classification of nets of conics in $\PG(2,q)$ with empty bases. In the following two corollaries,  $\mathcal{P}_i$ represents  the $\PGL(3,q)$-orbit of  pencils of conics associated with the $K$-orbit of lines $o_i$ (see Tables \ref{tableoflineseven} and \ref{solidsodd}).

\begin{Corollary}\label{netsb0odd}
    For $q$ odd, a net of conics ${\mathcal{N}}$ in $\PG(2,q)$ that has an empty base and at least one singular conic is of the form: $\mathcal{N}=\langle \mathscr{P}, \mathcal{C}\rangle$ where $\mathscr{P}\in\mathscr{P}_{8,2}\cup\mathscr{P}_{14,2}\cup\mathscr{P}_{15,2}$ and $\mathcal{C}$ is a conic in $\PG(2,q)$ different from a double line. Note that $\mathscr{P}$ has an empty base and $\mathcal{N}$ has no double lines.
\end{Corollary}

\begin{Corollary}\label{netsb0even}
    For $q>2$ even, a net of conics ${\mathcal{N}}$ in $\PG(2,q)$ that has an empty base and at least one pair of real or conjugate imaginary lines is of the form: $\mathcal{N}=\langle \mathscr{P}, \mathcal{C}\rangle$ where $\mathscr{P}\in \mathscr{P}_{7}\cup\mathscr{P}_{13}\cup\mathscr{P}_{14}$ and $\mathcal{C}$ is a conic in $\PG(2,q)$. Note that $\mathscr{P}$ has an empty base.
\end{Corollary}

For future reference, we record the following results, which can be derived from Theorems \ref{[0,q+1,0,q^2]}, \ref{[0,1,0,q^2+q]} and \ref{Maind=3}.

\begin{Corollary}
    There is a unique $\PGL(3,q)$-orbit of nets of conics in $\PG(2,q)$ with conic distribution $[1,0,0,q^2+q]$ (resp. $[q+1,0,0,q^2]$) for $q>2$ even.
\end{Corollary}

\begin{Corollary}\label{[0,0,0,q^2+q+1]}
    There are two $\PGL(3,q)$-orbits of nets of conics in $\PG(2,q)$ with no singular conics for $q$ odd. Whereas, nets of conics with no singular conics in $\PG(2,q)$, $q>2$ even, form a unique $\PGL(3,q)$-orbit.
\end{Corollary}

\section{Comparison with the construction by Schmidt}\label{Comparison with known constructions of MSRD codes}

In \cite{Schmidt1,Schmidt}, additive SRD codes in $M_{n \times n}(\Fq)$ were studied by Schmidt as subgroups of the set $\mathscr{S}_n(\Fq)$ of symmetric bilinear forms on an $n$-dimensional vector space over $\Fq$ (referred to as  $(n,d)$-sets or $d$-codes). Motivated by their connections to classical coding theory, including Kerdock codes, $\mathbb{Z}_4$-Delsarte-Goethals codes and cyclic codes, MSRD codes in $M_{n \times n}(\Fq)$ were investigated as optimal-size subsets of $\mathscr{S}_n(\Fq)$. The main tools used for these studies  were association schemes of alternating bilinear forms for $q$ even and association schemes of symmetric bilinear forms for $q$ odd.
In \cite{Schmidt1}, a construction of $(n,n-2t)$-codes (resp. $(n,n-2t+1)$-codes) of sizes $q^{n(t+1)}$ (resp. $q^{(n+1)t}$) for an integer $0\leq t\leq \frac{n-1}{2}$ (resp. $1\leq t\leq \frac{n}{2}$)
was provided for $n-d$ even (resp. $n-d$ odd). We denote this construction by $Y(t,n,q)$.\\

With this notation, for $3-d$ even, $Y(0,3,q)$ and $Y(1,3,q)$ correspond to $\Fq$-linear MSRD codes in $M_{3 \times 3}(\Fq)$ of size $q^3$ and minimum distances $3$, and size $q^6$ and minimun distance $1$, respectively. Thus $Y(0,3,q)$ corresponds to a plane of constant rank $3$, while $Y(1,3,q)$ corresponds to the whole set of $3\times 3$ symmetric matrices over $\Fq$. For $3-d$ odd, $Y(1,3,q)$ is an $\Fq$-linear MSRD code in $M_{3 \times 3}(\Fq)$ (of size $q^4$) with minimum distance $2$, which corresponds to a solid in $\PG(5,q)$ disjoint from the Veronese surface. In terms of linear system of conics, $Y(0,3,q)$ corresponds to a net of conics in $\PG(2,q)$ with no singular conics, and $Y(1,3,q)$ with $d=2$ corresponds to a pencil of conics in $\PG(2,q)$ with  empty base. By Corollaries \ref{pencilintrepqodd} and \ref{pencilintrepqeven}, such a pencil has conic distribution: $[1,0,1,q-1]$, $[0,1,2,q-2]$, or $[0,0,1,q]$. Additionally, we refer to \cite{Schmidt} for a similar construction of optimal $(n,d)$-sets when $n-d$ is even.

\section{Final Remarks}
In general, $\mathbb{F}_q$-linear CSRD codes in $M_{n \times n}(\mathbb{F}_q)$ with minimum distance $d$ can be studied as subspaces of $\PG(N-1, q)$ that intersect the $(d-1)$-st secant variety of the Veronese variety $\mathcal{V}^{(d-1)}_{n}(\mathbb{F}_q)$ trivially (i.e. subspaces of minimum rank $d$) and are maximal with respect to this property. While every MSRD code is a CSRD code, obtaining the conditions under which the converse fails remains an open question, particularly for $n > 3$. Note that, the study of $\Fq$-linear CSRD codes is significant as it facilitates the study of $\Fq$-linear SRD codes that can be extended to the maximum size.\\

The geometric interpretation of the $\Fq$-linear CSRD codes in $M_{3 \times 3}(\Fq)$ and their connection to linear systems of conics in $\PG(2, q)$ as  presented in this paper, suggests a potential correspondence  between $\Fq$-linear CSRD codes in $M_{n \times n}(\Fq)$ and linear systems of quadrics in $\PG(n-1, q)$. This might yield intriguing links between classical coding theory and linear systems of quadrics in $\PG(n-1, q)$ through the lens of $\Fq$-linear MSRD codes in $M_{n \times n}(\Fq)$ (see \cite{Schmidt1,Schmidt}).

\section*{Acknowledgements}

The first author acknowledges the support of the \textit{Croatian Science Foundation}, project number HRZZ-UIP-2020-02-5713. The second author is supported in part by the Slovenian Research and Innovation Agency (ARIS) under research program P1-0285 and research project J1-50000.

\begin{appendices}

In this appendix we gather in Tables \ref{tableoflineseven} and  \ref{solidsodd} 
 representatives of the $15$ $K$-orbits of lines in $\PG(5,q)$,  their point-orbit distributions and hyperplane-orbit distributions from \cite{lines,webs}. In Tables \ref{solidseven} and \ref{solidsodd}, we list representatives of the $15$ $K$-orbits of solids in $\PG(5,q)$, their point-orbit distributions and hyperplane-orbit distributions from \cite{solidsqeven,lines,webs}. Notation is as defined in Section~\ref{pre}; in particular, notation in the second column is as in \eqref{egSolid}.

\begin{table}[!htbp]
\begin{center}
\scriptsize
\begin{tabular}[h]{l ll ll} 
 \toprule
$L^K$& Representatives&$OD_0(L)$& $OD_4(L)$\\ \midrule \vspace{3pt}
$o_5$&$\begin{bmatrix}x&\cdot&\cdot\\\cdot&y&\cdot\\ \cdot&\cdot&\cdot\end{bmatrix}$ &$[2,0,q-1,0]$& $[1,2q^2+q,0,q^3-q^2]$   \\\vspace{3pt}
$o_6$&$\begin{bmatrix}x&y&\cdot\\y&\cdot&\cdot\\ \cdot&\cdot&\cdot\end{bmatrix}$  &$[1,1,q-1,0]$& $[q+1, \frac{3q^2+q}{2},\frac{q^2-q}{2}, q^3-q^2]$   \\\vspace{3pt}
$o_{8,1}$ & $\begin{bmatrix}x&\cdot&\cdot\\\cdot&y&\cdot\\ \cdot&\cdot&-y\end{bmatrix}$&$[1,0,1,q-1]$& $[1,q^2+\frac{3}{2}q,\frac{q}{2},q^3-q]$ 
   \\\vspace{3pt}
$o_{8,3}$ &$\begin{bmatrix}x&\cdot&\cdot\\\cdot&\cdot&y\\ \cdot&y&\cdot\end{bmatrix}$ &$[1,1,0,q-1]$& $[q+1,q^2+q,0,q^3-q]$    \\ \vspace{3pt}
$o_9$ & $\begin{bmatrix}x&\cdot&y\\\cdot&y&\cdot\\ y&\cdot&\cdot\end{bmatrix}$ &$[1,0,0,q]$& $[1,q^2+q,0,q^3]$    \\\vspace{3pt}
$o_{10}$ & $\begin{bmatrix}v_0x&y&\cdot\\y&x+uy&\cdot\\ \cdot&\cdot&\cdot\end{bmatrix}$ &$[0,0,q+1,0]$& $[1,q^2+q,q^2,q^3-q^2]$     \\\vspace{3pt}
$o_{12,1}$ &$\begin{bmatrix}\cdot&x&\cdot\\x&\cdot&y\\ \cdot&y&\cdot\end{bmatrix}$&$[0,q+1,0,0]$& $[q^2+q+1,\frac{q^2+q}{2},\frac{q^2-q}{2},q^3-q^2]$     \\ \vspace{3pt}
$o_{12,3}$ & $\begin{bmatrix}\cdot&x&\cdot\\x&x+y&y\\ \cdot&y&\cdot\end{bmatrix}$&$[0,1,q,0]$&   $[q+1,q^2+\frac{q}{2}, q^2-\frac{q}{2}, q^3-q^2]$   \\ \vspace{3pt}
$o_{13,1}$ &$\begin{bmatrix}\cdot&x&\cdot\\x&y&\cdot\\ \cdot&\cdot&-y\end{bmatrix}$ &$[0,1,1,q-1]$& 
 $[q+1,\frac{q^2}{2}+q,\frac{q^2}{2},q^3-q]$    \\ \vspace{3pt}
$o_{13,3}$ &$\begin{bmatrix}\cdot&x&\cdot\\x&x+y&\cdot\\ \cdot&\cdot&y\end{bmatrix}$ &$[0,0,2,q-1]$&  $[1,\frac{q^2+3q}{2},\frac{q^2+q}{2},q^3-q]$   \\ \vspace{3pt}
$o_{14,1}$ & $\begin{bmatrix}x&\cdot&\cdot\\\cdot&-(x+y)&\cdot\\ \cdot&\cdot&y\end{bmatrix}$&$[0,0,3,q-2]$&  $[1,\frac{q^2}{2}+2q,\frac{q^2}{2}+q,q^3-2q]$   \\\vspace{3pt}
$o_{15,1}$ &$\begin{bmatrix}v_1y&x&\cdot\\x&ux+y&\cdot\\ \cdot&\cdot&x\end{bmatrix}$ &$[0,0,1,q]$& $[1,\frac{q^2}{2}+q,\frac{q^2}{2},q^3]$     \\ \vspace{3pt}
$o_{16,1}$ & $\begin{bmatrix}\cdot&\cdot&x\\\cdot&x&y\\ x&y&\cdot\end{bmatrix}$&$[0,1,0,q]$& $[q+1,\frac{q^2+q}{2}, \frac{q^2-q}{2},q^3]$ \\\vspace{3pt}
$o_{16,3}$ &$\begin{bmatrix}\cdot&\cdot&x\\\cdot&x&y\\ x&y&y\end{bmatrix}$& $[0,0,1,q]$&  $[1,\frac{q^2}{2}+q,\frac{q^2}{2},q^3]$  \\  \vspace{3pt}
$o_{17}$ & $\begin{bmatrix}\alpha^{-1}x&y&\cdot\\y&\beta y-\gamma x&x\\ \cdot&x&y\end{bmatrix}$&$[0,0,0,q+1]$&  $[1,\frac{q^2+q}{2},\frac{q^2-q}{2},q^3+q]$  \\ \bottomrule

 \end{tabular}
 \caption{\label{tableoflineseven}\scriptsize The $15$ $K$-orbits of lines in $\PG(5,q)$ defined in \cite{lines}, their representatives, point-orbit distributions and hyperplane-orbit distributions for $q>2$ even.  The parameters $u,v_0,v_1,\alpha, \beta, \gamma$ in $\Fq$ are defined as follows: $v_i\lambda^2+uv_i\lambda-1 \neq 0$ for all $\lambda\in \Fq$ and $i \in\{0,1\}$ and
$\lambda^3 + \gamma \lambda^2 - \beta \lambda + \alpha \neq 0$ for all $\lambda \in \Fq$.}
\end{center}
\end{table}

\begin{table}[!ht]
\center
\scriptsize
\begin{tabular}{lllll}
\toprule
$S^K$ & Representatives&$OD_0(S)$& $OD_4(S)$
 & Conditions \\
\midrule 
\vspace{3pt}
$\Omega_1$ & 
$\begin{bmatrix} x&y&z\\y&t&\cdot\\z&\cdot&t \end{bmatrix}$ &    $[1,q+1,2q^2-1,q^3-q^2]$ & $[1,q/2,q/2,0]$
& \\
\vspace{3pt}
$\Omega_2$ & 
$\begin{bmatrix} x&y&z\\y&t&\cdot\\z&\cdot&\cdot \end{bmatrix}$   & $[q+1,q+1,2q^2-q-1,q^3-q^2]$ & $[1,q,0,0]$
& \\
\vspace{3pt}
$\Omega_3$ & 
$\begin{bmatrix} x&y&z\\y&\cdot&t\\z&t&\cdot \end{bmatrix}$  &   $[1,q^2+q+1,q^2-1,q^3-q^2]$ & $[q+1,0,0,0]$
& \\
\vspace{3pt}
$\Omega_4$ & 
$\begin{bmatrix} x&\cdot&y\\\cdot&z&\cdot\\y&\cdot&t \end{bmatrix}$  &   $[q+2,1,2q^2-2,q^3-q^2]$ & $[0,q+1,0,0]$
& \\
\vspace{3pt}
$\Omega_5$ & 
$\begin{bmatrix} \cdot&x&y\\x&z&t\\y&t&x \end{bmatrix}$  &    $[1,q+1,q^2-1,q^3]$ & $[1,0,0,q]$
& \\
\vspace{3pt}
$\Omega_6$ & 
$\begin{bmatrix} x&\cdot&y\\\cdot&z&t\\y&t&\cdot \end{bmatrix}$  &   $[2,q+1,q^2+q-2,q^3-q]$ & $[1,1,0,q-1]$ 
& \\
\vspace{3pt}
$\Omega_7$ & 
$\begin{bmatrix} x&y&z\\ y&x+\gamma y&t\\z&t&y \end{bmatrix}$   &     $[0,q+1,q^2+q,q^3-q]$ & $[1,0,1,q-1]$ 
& $\operatorname{Tr}(\gamma^{-1})=1$ \\
\vspace{3pt}
$\Omega_8$ & 
$\begin{bmatrix}x&y&z\\y&t&z\\z&z&y\end{bmatrix}$  &       $[3,1,q^2+2q-3,q^3-q]$ & $[0,2,0,q-1]$
& \\
\vspace{3pt}
$\Omega_9$ & 
$\begin{bmatrix}x&x&y\\x&z&t\\y&t&t\end{bmatrix}$ &      $[4,1,q^2+3q-4,q^3-2q]$ & $[0,3,0,q-2]$
& \\
\vspace{3pt}
$\Omega_{10}$ & 
$\begin{bmatrix}x&y&z\\y& y+\gamma t& t\\ z& t&y\end{bmatrix}$  &        $[1,1,q^2+2q-1,q^3-q]$ & $[0,1,1,q-1]$
& $\operatorname{Tr}(\gamma^{-1})=1$ \\
\vspace{3pt}
$\Omega_{11}$ & 
$\begin{bmatrix} x&y&z\\y&t&\cdot\\z&\cdot&y \end{bmatrix}$ &       $[2,1,q^2+q-2,q^3]$ & $[0,1,0,q]$
& \\
\vspace{3pt}
$\Omega_{12}$ & 
$\begin{bmatrix}x&y&z\\y&t&\gamma y+z\\z&\gamma y+z&y\end{bmatrix}$  &     $[2,1,q^2+q-2,q^3]$ & $[0,1,0,q]$ 
& $\operatorname{Tr}(\gamma^{-1})=1$ \\
\vspace{3pt}
$\Omega_{13}$ & 
$\begin{bmatrix} x&y&z\\y&\gamma x+y&t\\z&t&\gamma x+z \end{bmatrix}$ &     $[0,1,q^2+3q,q^3-2q]$ & $[0,1,2,q-2]$
& $\operatorname{Tr}(\gamma)=1$ \\
\vspace{3pt}
$\Omega_{14}$ & 
$\begin{bmatrix} x &y&\gamma x+y+\gamma t\\y&\gamma x+y&z\\\gamma x+y+\gamma t&z&t \end{bmatrix}$ &     $[0,1,q^2+q,q^3]$ & $[0,0,1,q]$ 
& $\operatorname{Tr}(\gamma)=1$ \\
\vspace{3pt}
$\Omega_{15}$ & 
$\begin{bmatrix}x&y&bz+cy\\y&z&t\\bz+cy&t&y\end{bmatrix}$ 
&          $[1,1,q^2-1,q^3+q]$ & $[0,0,0,q+1]$
& $b\lambda^3 + c\lambda + 1$ irreducible over $\bF_q$ \\
\bottomrule
\end{tabular}
\caption{\scriptsize \label{solidseven}The $15$ $K$-orbits of solids in $\PG(5,q)$ defined in \cite{solidsqeven}, their representatives, point-orbit distributions and hyperplane-orbit distributions for $q>2$ even. $\operatorname{Tr}$ denotes the trace map from $\mathbb{F}_q$ to $\mathbb{F}_2$.}

\end{table}

\begin{table}[!htbp]
\begin{center}
\scriptsize
\begin{tabular}[h]{ l l l l   l} 
 \toprule

 $S^K/L^K$ &Representatives  & Representatives & $OD_0(S)=OD_4(L)$& $OD_4(S)=OD_0(L)$  \\
 &of $S^K$&of $L^K$&&\\ \midrule \\
   \vspace{3pt}
 $ \Omega_5/o_5$  &		$\begin{bmatrix} .&x&y\\x&.&z\\y&z&t           \end{bmatrix}$ & 
 $\begin{bmatrix}x&\cdot&\cdot\\\cdot&y&\cdot\\ \cdot&\cdot&\cdot\end{bmatrix}$
 &$[1,2q^2+q,0,q^3-q^2]$     &$[2,\frac{q-1}{2},\frac{q-1}{2},0]$
 
 \\
  \vspace{3pt}

$\Omega_6/o_6$  &$\begin{bmatrix} .&.&x\\.&y&z\\x&z&t           \end{bmatrix}$ &
$\begin{bmatrix}x&y&\cdot\\y&\cdot&\cdot\\ \cdot&\cdot&\cdot\end{bmatrix}$ 
&$[q+1, \frac{3q^2+q}{2},\frac{q^2-q}{2}, q^3-q^2]$   &$[1,q,0,0]$
 \\ 

\vspace{3pt}

$\Omega_{8,1}/o_{8,1}$  &$\begin{bmatrix} .&x&y\\x&z&t\\y&t&z        \end{bmatrix}$ &
$\begin{bmatrix}x&\cdot&\cdot\\\cdot&y&\cdot\\ \cdot&\cdot&-y\end{bmatrix}$
&$[2, q^2+\frac{3q-1}{2},\frac{q-1}{2}, q^3-q]$    & $[1,1,0,q-1]$
 \\

\vspace{3pt}

 $\Omega_{8,2}/o_{8,2}$  &$\begin{bmatrix} .&x&y\\x&\delta z&t\\y&t&z           \end{bmatrix}$

 &$\begin{bmatrix}x&\cdot&\cdot\\\cdot&y&\cdot\\ \cdot&\cdot& -\delta y\end{bmatrix}$ 
 & $[0,q^2+\frac{3q+1}{2},\frac{q+1}{2},q^3-q]$   &$[1,0,1,q-1]$
 \\

\vspace{3pt}

  $\Omega_{9}/o_9$  &$\begin{bmatrix} .&x&y\\x&-y&z\\y&z&t           \end{bmatrix}$ &
  $\begin{bmatrix}x&\cdot&y\\\cdot&y&\cdot\\ y&\cdot&\cdot\end{bmatrix}$ 
  &$[1,q^2+q,0,q^3]$    &$[1,0,0,q]$ \\

\vspace{3pt}

  $\Omega_{10}/o_{10}$  &$\begin{bmatrix} x&uv_0x&y\\uv_0x&-v_0x&z\\y&z&t          \end{bmatrix}$  &
  $\begin{bmatrix}v_0x&y&\cdot\\y&x+uy&\cdot\\ \cdot&\cdot&\cdot\end{bmatrix}$ 
  &$[1,q^2+q,q^2,q^3-q^2]$   &$[0,\frac{q+1}{2},\frac{q+1}{2},0]$
  
  \\

 \vspace{3pt}

$\Omega_{12,1}/o_{12,1}$ &  $\begin{bmatrix} x&.&y\\.&z&.\\y&.&t           \end{bmatrix}$ &
$\begin{bmatrix}\cdot&x&\cdot\\x&\cdot&y\\ \cdot&y&\cdot\end{bmatrix}$
&$[q+2,q^2+\frac{q-1}{2}, q^2-\frac{q+1}{2}, q^3-q^2]$   & $[0,q+1,0,0]$
\\

\vspace{3pt}

$\Omega_{13,1}/o_{13,1}$ &$\begin{bmatrix} x&.&y\\.&z&t\\y&t&z          \end{bmatrix}$  &
$\begin{bmatrix}\cdot&x&\cdot\\x&y&\cdot\\ \cdot&\cdot&-y\end{bmatrix}$
& $[3,\frac{q^2+3q-2}{2},\frac{q^2+q-2}{2},q^3-q]$   &$[0,2,0,q-1]$ 
\\

\vspace{3pt}

$\Omega_{13,2}/o_{13,2}$ &  $\begin{bmatrix} x&.&y\\ .&\delta z&t\\y&t&z \end{bmatrix}$ &  
$\begin{bmatrix}\cdot&x&\cdot\\x&y&\cdot\\ \cdot&\cdot&-\delta y\end{bmatrix}$ 
&$[1,\frac{q^2+3q}{2},\frac{q^2+q}{2},q^3-q]$  &$[0,1,1,q-1]$
 \\

\vspace{3pt}

$\Omega_{14,1}/o_{14,1}$ & $\begin{bmatrix} x&y&z\\y&x&t\\z&t&x           \end{bmatrix}$  & 
$\begin{bmatrix}x&\cdot&\cdot\\\cdot&-(x+y)&\cdot\\ \cdot&\cdot&y\end{bmatrix}$
&$[4,\frac{q^2-1}{2}+2q-1,\frac{q^2-1}{2}+q-1,q^3-2q]$    &$[0,3,0,q-2]$
\\ 

\vspace{3pt}

$\Omega_{14,2}/o_{14,2}$ &  
$\begin{bmatrix} \delta x&y&z\\y&x&t\\z&t&\delta x           \end{bmatrix}$ &
$\begin{bmatrix}x&\cdot&\cdot\\\cdot&-\delta(x+y)&\cdot\\ \cdot&\cdot&y\end{bmatrix}$
&$[0,\frac{q^2+1}{2}+2q,\frac{q^2+1}{2}+q,q^3-2q]$     &$[0,1,2,q-2]$
\\  

\vspace{3pt}

$\Omega_{15,1}/o_{15,1}$ &  $\begin{bmatrix} x&y&z\\y&-v_1x&t\\z&t&-y+uv_1x          \end{bmatrix}$ &  
$\begin{bmatrix}v_1y&x&\cdot\\x&ux+y&\cdot\\ \cdot&\cdot&x\end{bmatrix}$

&$[2,\frac{q^2-1}{2}+q,\frac{q^2-1}{2},q^3]$   &$[0,1,0,q]$ 
 \\

\vspace{3pt}

$\Omega_{15,2}/o_{15,2}$ &  $\begin{bmatrix} x&y&z\\y&-v_2x&t\\z&t&-y+uv_2x          \end{bmatrix}$ &  
$\begin{bmatrix}v_2y&x&\cdot\\x&ux+y&\cdot\\ \cdot&\cdot&x\end{bmatrix}$
&$[0,\frac{q^2+1}{2}+q,\frac{q^2+1}{2},q^3]$   &$[0,0,1,q]$
 \\

 \vspace{3pt}
$\Omega_{16,1}/o_{16,1}$ & $\begin{bmatrix} x&y&z\\y&-z&.\\z&.&t           \end{bmatrix}$ & 
$\begin{bmatrix}\cdot&\cdot&x\\\cdot&x&y\\ x&y&\cdot\end{bmatrix}$
&$[2,\frac{q^2-1}{2}+q,\frac{q^2-1}{2},q^3]$     & $[0,1,0,q]$
   \\ 

\vspace{3pt}

$\Omega_{17}/o_{17}$ &  $\begin{bmatrix} \alpha\gamma z-\alpha t&x&y\\x&z&t\\y&t&-x-\beta z \end{bmatrix}$ &
$\begin{bmatrix}\alpha^{-1}x&y&\cdot\\y&\beta y-\gamma x&x\\ \cdot&x&y\end{bmatrix}$

&  $[1,\frac{q^2+q}{2},\frac{q^2-q}{2},q^3+q]$   &$[0,0,0,q+1]$
\\

\bottomrule
 \end{tabular}
 \caption{\label{solidsodd}{\scriptsize The $15$ $K$-orbits of solids and lines in $\PG(5,q)$ defined in \cite{lines,webs}, their representatives, point-orbit distributions and hyperplane-orbit distributions for $q$ odd. The parameters $\delta$, $u$, $v_0$, $v_1$ and $v_2$ in are defined such that  $\delta\not\in\square_q$, $ v_i\lambda^2+uv_i\lambda-1 \neq 0$ for all $\lambda\in \Fq$ and $i\in \{0,1,2\}$, $-v_1\in\square_q$ and $-v_2\not\in\square_q$. The parameters $\alpha$, $\beta$ and $\gamma$ are defined such that 
 $\lambda^3 + \gamma \lambda^2 - \beta \lambda + \alpha \neq 0$ for all $\lambda \in \Fq$.}}

\end{center}

\end{table}

\end{appendices}

\end{document}